\documentclass[twoside,11pt,reqno]{amsart}
\usepackage{amsmath,amssymb,amscd,mathrsfs}
\usepackage{pb-diagram}

\makeatletter

\@addtoreset{equation}{section}

\input prepictex
\input pictex
\input postpictex

\newtheorem{Proposition}{Proposition}[section]
\newtheorem{Lemma}[Proposition]{Lemma}
\newtheorem{Theorem}[Proposition]{Theorem}
\newtheorem{Corollary}[Proposition]{Corollary}

\newtheorem{Problem}[Proposition]{Problem}
\newtheorem{Conjecture}[Proposition]{Conjecture}
\newtheorem{Hypothesis}[Proposition]{Hypothesis}

\newbox\squ  % box character for ends of proofs
\setbox\squ=\hbox{\vrule width.6pt
   \vbox{\hrule height.6pt width.4em\kern1ex
         \hrule height.6pt}%
   \vrule width.6pt}

\def\FF{\mathbb{F}}

\def\Rep#1{\operatorname{Rep}(#1)}

\def\Proj#1{\operatorname{Proj}(#1)}
\def\Mod#1{#1\!\operatorname{-Mod}}

\def\CH{{\operatorname{ch}_q\:}}

\def\Q{{\mathbb Q}}
\def\Z{{\mathbb Z}}

\def\0{{\bar 0}}
\def\1{{\bar 1}}

\def\HOM{{\operatorname{HOM}}}
\def\Hom{{\operatorname{Hom}}}

\def\Laurent{\mathcal A}

\def\qdim{{\operatorname{dim}_q}\,}

\def\Ind{{\operatorname{Ind}}}

\def\Res{{\operatorname{Res}}}

\def\ch{{\operatorname{ch}\:}}

\def\height{{\operatorname{ht}}}

\def\bi{\text{\boldmath$i$}}
\def\bj{\text{\boldmath$j$}}

\def\bk{\text{\boldmath$k$}}
\def\bl{\text{\boldmath$l$}}
\def\eps{{\varepsilon}}
\def\phi{{\varphi}}
\def\emptyset{{\varnothing}}
\def\ga{{\gamma}}
\def\Ga{{\Gamma}}

\def\de{{\delta}}
\def\De{{\Delta}}
\def\al{{\alpha}}
\def\be{{\beta}}

\def\si{{\sigma}}

\def\iso{\,\tilde\rightarrow\,}
\def\underbar{\mathpalette\@underbar}
\def\h{{\mathfrak h}}
\def\words{{\mathbf W}}
\def\onto{{\twoheadrightarrow}}
\def\into{{\hookrightarrow}}

\def\@underbar#1#2{\settowidth{\@tempdimb}{$#1#2$}\@tempdimb=0.8\@tempdimb
                   \ooalign{$#1#2$\crcr%
                         \hfil\rule[-.5mm]{\@tempdimb}{.4pt}\hfil}}

\newdimen\hoogte    \hoogte=11.5pt    % hoogte  van hokje
\newdimen\breedte   \breedte=11.5pt   % breedte van hokje
\newdimen\dikte     \dikte=0.5pt    % dikte lijn

\newenvironment{young}{\begingroup
       \def\vr{\vrule height0.8\hoogte width\dikte depth 0.2\hoogte}
       \def\fbox##1{\vbox{\offinterlineskip
                    \hrule height\dikte
                    \hbox to \breedte{\vr\hfill##1\hfill\vr}
                    \hrule height\dikte}}
       \vbox\bgroup \offinterlineskip \tabskip=-\dikte \lineskip=-\dikte
            \halign\bgroup &\fbox{##\unskip}\unskip  \crcr }
       {\egroup\egroup\endgroup}
\def\diagram#1{\relax\ifmmode\vcenter{\,\begin{young}#1\end{young}\,}\else%
              $\vcenter{\,\begin{young}#1\end{young}\,}$\fi}

\begin{document}

\title[Representations of KLR Algebras]{Representations of Khovanov-Lauda-Rouquier  Algebras and Combinatorics of Lyndon Words}
\author{Alexander Kleshchev and Arun Ram}

\begin{abstract}
We construct irreducible representations of affine Khovanov-Lauda-Rouquier algebras of arbitrary finite type. The irreducible representations arise as simple heads of appropriate induced modules, and thus our construction is similar to that of Bernstein and Zelevinsky for affine Hecke algebras of type $A$. The highest weights of irreducible modules are given by the so-called good words, and the highest weights of the `cuspidal modules' are given by the good Lyndon words. In a sense, this has been predicted by Leclerc. 
\end{abstract}
\thanks{{\em 2000 Mathematics Subject Classification:} 20C08.}
\thanks{The first author is supported by NSF grant %no. 
DMS-0654147. The second author is supported by
NSF Grant DMS-0353038
and Australian Research Council Grants DP0986774
and DP0879951.}
\address{Department of Mathematics, University of Oregon, Eugene, Oregon, USA.}
\email{klesh@uoregon.edu}
\address{Department of Mathematics and Statistics,  
The University of Melbourne, 
Parkville, VIC, 3010,  
Australia and Department of Mathematics, 
University of Wisconsin--Madison, 
Madison, WI 53706}
\email{aram@unimelb.edu.au}
\maketitle

\section{Introduction}\label{SIntro}

Khovanov and Lauda \cite{KL1,KL2} and Rouquier \cite{R} have introduced a new family of graded algebras whose representation theory is related to categorification of quantum groups. In this paper we construct irreducible representations of Khovanov-Lauda-Rouquier algebras of finite type. 
%Following Rouquier, we actually refer to these algebras as the {\em quiver nil Hecke algebras}. 

The irreducible representations of Khovanov-Lauda-Rouquier algebras arise as simple heads of appropriate induced modules, and thus our construction is similar to that of Bernstein and Zelevinsky \cite{BZ,Z} for affine Hecke algebras of type $A$. In type $A$ this is not surprising in view of \cite{BKyoung} and \cite[\S3.2.6]{R}. However, our parameterization of irreducible modules depends on the choice of order of simple roots, which can be arbitrary. In this sense our results are more general than the usual Bernstein-Zelevinsky classification even in type $A$. 

The highest weights of irreducible modules are given by the so-called good words, and the highest weights of the `cuspidal modules' are given by the good Lyndon words. In a sense, this has been predicted by Leclerc~\cite{Lec}. 

Our classification works over fields of arbitrary characteristic. In fact, we conjecture that the formal characters of irreducible modules over Khovanov-Lauda-Rouquier algebras of finite types do not depend on the characteristic of the ground field, see Conjecture~\ref{Conj}. 

An alternative {\em inductive} approach to classification of irreducible modules over Khovanov-Lauda-Rouquier algebras is suggested by Lauda and Vazirani \cite{LV}. They define crystal operations on irreducible modules in terms of induction and restriction functors and identify the resulting crystals. An irreducible module is then characterized through its branching properties to smaller algebras whose irreducible modules are known by  induction. 

\section{Preliminaries}

\subsection{Cartan datum} A {\em Cartan datum} is a pair $(I,\cdot)$ consisting of a set $I$ and a $\Z$-valued symmetric bilinear form $i,j\mapsto i\cdot j$ on the free abelian group $\Z[I]$ such that $i\cdot i\in \{2,4,6,\dots\}$ for all $i\in I$ and 
$2(i\cdot j)/(i\cdot i)\in\{0,-1,-2\dots\}$ for all $i\neq j$ in $I$. Let 
$Q_+ := \bigoplus_{i \in I} \Z_{\geq 0} i$. For $\alpha \in Q_+$, we write $\height(\alpha)$ for the sum of its 
coefficients when expanded in terms of the $i$'s.

Set
$$
a_{ij}:=2(i\cdot j)/(i\cdot i)\qquad(i,j\in I).
$$
If  the {\em Cartan matrix} $A:=(a_{ij})_{i,j\in I}$ has finite type, see \cite[\S 4]{Kac}, then we also say that the Cartan datum is of finite type. 
Following \cite[\S 1.1]{Kac}, let $(\h,\Pi,\Pi^\vee)$ be a realization of the Cartan matrix $A$, so we have simple roots $\{\al_i\mid i\in I\}$. We identify $\al_i$ and $i$. 
%, the fundamental dominant weights $\{\La_i\mid i\in I\}$, and the normalized invariant form $(\cdot,\cdot)$ such that
%$$(\al_i,\al_j)=a_{i,j}, \quad (\La_i,\al_j)=\de_{i,j}\qquad(i,j\in I).$$
Denote by $\De_+\subset Q_+$ the set of {\em positive}\, roots, cf. \cite[\S 1.3]{Kac}. 

\subsection{Ground rings and parameters}\label{SSGRP}
Throughout the paper $\FF$ is an arbitrary field. 

Let $q$ be indeterminate and $\Q(q)$ the field of rational fractions. We denote $\Laurent:=\Z[q,q^{-1}]$. Let\, $\bar{ }:\Q(q)\to \Q(q)$ be the $\Q$-algebra involution with $\bar q=q^{-1}$, referred to as the {\em bar-involution}. We have  $\bar \Laurent=\Laurent$. 

Given $\beta\in \Z[I]$, denote
%$i\in I$, denote
\begin{align*}
q_\beta&:=q^{(\beta\cdot \beta)/2},\qquad 
[n]_\beta:=(q_\beta^{n}-q_\beta^{-n})/(q_\beta-q_\beta^{-1}),\\ 
 [n]^!_\beta&:=[n]_\beta[n-1]_\beta\dots[1]_\beta, \qquad 
\left[
\begin{matrix}
 n   \\
 m
\end{matrix}
\right]_\beta
=\frac{[n]^!_\beta}{[n-m]^!_\beta[m]^!_\beta}.
\end{align*}
In particular, for $i\in I$, we have $[n]_i,[n]_i^!,\left[
\begin{matrix}
 n   \\
 m
\end{matrix}
\right]_i$. 

Let $A$ be a $Q_+$-graded $\Q(q)$-algebra, $\theta\in A_{\al}$ for $\al\in Q_+$, and $n\in\Z_{\geq 0}$. We use the standard notation for  %corresponding 
quantum divided powers:   $\theta^{(n)}:= \theta^n/[n]_\al^!.$

\subsection{Words}
Denote by $$\words:=\bigsqcup_{d\geq 0} I^d$$ the set of all tuples $\bi=(i_1,\dots,i_d)$ of elements of $I$, which we refer to as {\em words}. We consider $\words$ as a monoid with respect to the concatenation product: $\bi\bj=(i_1,\dots,i_d,j_1,\dots,j_{d'})$ for $\bi=(i_1,\dots,i_d)$ and $\bj=(j_1,\dots,j_{d'})$, the empty word $\emptyset$ being the identity in $\words$. 

If $\bi\in\words$, we can write it in the form 
$
\bi=(j_1)^{m_1}\dots (j_r)^{m_r}
$
for $j_1,\dots,j_r\in I$ such that $j_s\neq j_{s+1}$ for all $s=1,2,\dots,r-1$. 
We then denote
$$
[\bi]!:=[m_1]^!_{j_1}\dots[m_r]^!_{j_r}. 
$$
For example if $\bi=(1,1,2,2,2,1)=(1)^2(2)^3(1)$  then $[\bi]!=[2]^!_1[3]^!_2$. 

For $\bi=(i_1,\dots,i_d)$ denote %fine the {\em type} of $\bi$ to be 
$$
|\bi|:=\al_{i_1}+\dots+\al_{i_d}\in Q_+.
$$
The symmetric group $S_d$ 
with simple transpositions $s_1,\dots,s_{d-1}$ %. The group $S_d$ 
acts  
on $I^d$ on the left by place
permutations. The $S_d$-orbits on $I^d$ are the sets
\begin{equation*}
\words^\alpha := \{\bi\in I^d \:|\:|\bi| = \alpha\}
\end{equation*}
parametrized by the elements $\alpha \in Q_+$ of height $d$. 

%Given $\bj\in\words$, we denote $[n]_\bj:=(q^{n(|\bj|\cdot |\bj|)/2}-q^{-n(|\bj|\cdot |\bj|)/2})/(q^{(|\bj|\cdot |\bj|)/2}-q^{-(|\bj|\cdot |\bj|)/2})$,  $ [n]^!_{\bj}:=[n]_{\bj}[n-1]_{\bj}\dots[1]_{\bj}$, etc. 

%$[n]_\bj^!:=[n]_i^!$, where $i\in I$ is determined from $|\bj|\cdot |\bj|=i\cdot i$. 

%Throughout the paper we assume that the {\em Cartan matrix} $A=(a_{ij})_{i,j\in I}$ is of {\em finite type}. 
%As the set $I$ is finite, we will assume that
%$$I=\{1,2,\dots,N\}.$$
%In particular, `$i<j$' makes sense for any $i\neq j$ in $I$. 

%Together wit the Cartan matrix $A$ there come the following standard notions:
%\begin{enumerate}
%\item[$\De$] the corresponding root system;
%\item[$\De_+$]$=\{\be_1,\dots,\be_n\}$ (a fixed choice of) the set of positive roots; we use a specific order on the positive roots, see (\ref{EOrder});
%\item[$\Pi$]$=\{\al_i\mid i\in I\}$  a set of simple roots in $\De$; 
%\item[$Q$]$=\bigoplus_{i\in I}\Z\al_i$  the root lattice;
%\item[$Q_+$]$=\bigoplus_{i\in I}\Z_{\geq 0}\al_i$;
%\item[$I$]$=\{1,2,\dots,\ell\}$;
%\item[$\g$]  finite dimensional simple Lie algebra over $\C$ with root system $\De$;
%\item[$\n$]  the positive maximal nilpotent subalgebra of $\g$ corresponding to $\De_+$;
%\item[$U_q(\n)$] the quantized enveloping algebra of $\n$ with the set of Chevalley generators $\{E_i\mid i\in I\}$.
%\end{enumerate}
%For $\al=\sum_{i\in I}m_i\al_i\in Q_+$ %and $\La\in P$ define the {\em height} of $\al$ as $\height(\al):=\sum_{i\in I}m_i.$
%,\qquad \ell(\La):=\sum_{i\in I}(\La,\al_i).$

%Fix a total order `$<$' on $I$. 

\subsection{Quivers with automorphism}\label{QA}  
Following \cite[\S 12,14]{Lubook} and 
\cite[\S3.2.4]{R}, %let $\Ga$ be 
a {\em graph with a compatible automorphism} consists of the following data 
\begin{enumerate}
\item[$\bullet$] a set $\tilde I$ (vertices);
\item[$\bullet$] a set $H$ (edges) and a map with finite fibers $h\mapsto[h]$ from $H$ to the set of two-element subsets of $\tilde I$;
%\item[$\bullet$] maps $s:H\to \tilde I$ (source) and $t:H\to \tilde I$ (target) such that $\{s(h),t(h)\}=[h]$ for all $h\in H$;
\item[$\bullet$] automorphisms $a:\tilde I\to \tilde I$ and $a:H\to H$ such that for any $h\in H$ we have  $[a(h)]=a([h])$  as subsets of $\tilde I$, and such that there is no edge between two vertices in the same $a$-orbit. 
%$s(a(h))=a(s(h))$, $t(a(h))=a(t(h))$, and $s(h)$ and $t(h)$ are not in the same $a$-orbit.
\end{enumerate}

Set $I:=\tilde I/a$, and let $i,j\in I$. Define $i\cdot i$ to be the cardinality of the orbit $i$, and $i\cdot j$ to be the negative of the number of edges joining some vertex in the orbit $i$ with some vertex in the orbit $j$. By \cite[\S13.2.9]{Lubook}, $(I,\cdot)$ is a Cartan datum. By \cite[Proposition 14.1.2]{Lubook}, every Cartan datum arises in this way. 

%In this paper we always assume that the Cartan datum $(I,\cdot)$ is of {\em finite type}. Then the graph $\Gamma$ and the automorphism $a$ can be chosen in a canonical way as in \cite[\S\S 14.1.4,14.1.6]{Lubook}. 

A {\em quiver with a compatible automorphism} is a graph with a compatible automorphism as above together with  maps $s:H\to \tilde I$ (source) and $t:H\to \tilde I$ (target) such that $\{s(h),t(h)\}=[h]$ and $s(a(h))=a(s(h))$, $t(a(h))=a(t(h))$ for all $h\in H$. Let $i,j\in I=\tilde I/a$. 
Set 
$$d_{ij}:= \big|\{h\in H\mid s(h)\in i\ \text{and}\ t(h)\in j\}/a\big|$$
and 
$m(i,j):=\operatorname{LCM}(i\cdot i,j\cdot j)$. 
 It is noted in \cite[\S3.2.4]{R} that
\begin{equation}\label{EDIJDJI}
d_{ij}+d_{ji}=-2(i\cdot j)/m(i,j)\qquad(i\neq j). 
\end{equation}

\section{Khovanov-Lauda-Rouquier algebras}
\subsection{Preliminary data}
%Let $\FF$ be an arbitrary commutative ring with $1$. 
Let $\Gamma$ be a quiver with a compatible automorphism as in \S\ref{QA}. The quiver $\Gamma$ determines  the set of polynomials 
$$\{Q_{ij}(u,v)\in \FF[u,v]\mid i,j\in I\}$$ 
as follows. If $i=j$ take $Q_{ij}(u,v)=0$. 
For $i\neq j$, define the polynomial 
$Q_{ij}(u,v)\in \FF[u,v]$
as follows:
\begin{equation}\label{EQIJ}
Q_{ij}(u,v):=(-1)^{d_{ij}}(u^{m(i,j)/(i\cdot i)}-v^{m(i,j)/(j\cdot j)})^{-2(i\cdot j)/m(i,j)}.
\end{equation}

For example, if $i\cdot j=0$ we have $Q_{i,j}(u,v)=1$. Let $i\cdot j\neq0$. Assume that $a_{ij}=-1$ or $a_{ji}=-1$ (this assumption always holds for Cartan matrices of finite type). 
In this case $m(i,j)=\max(i\cdot i,j\cdot j)$. So $-2(i\cdot j)/m(i,j)=1$, $m(i,j)/(i\cdot i)=-a_{ij}$, and $m(i,j)/(j\cdot j)=-a_{ji}$. Finally, in view of (\ref{EDIJDJI}), we have $d_{ij}+d_{ji}=1$. %, which dictates the opposite choice of signs in front of $Q_{ij}$ and $Q_{ji}$. 
In particular, for the case where the Cartan matrix $A$ is of finite type, the polynomials $Q_{ij}(u,v)$ 
%the extra data which comes from the orientation is equivalent to 
are determined by $A$ and a partial order `$\preceq$' on $I$ such that $i\prec j$ or $j\prec i$ whenever $i\cdot j<0$; more precisely, 
%the polynomials $Q_{ij}(u,v)$  are as follows: 
%chosen as follows: first pick a partial order `$\leq$' on $I$ such that $i<j$ or $j<i$ whenever $i\cdot j<0$; then set 
\begin{equation}\label{EArun}
Q_{ij}(u,v):=
\left\{
\begin{array}{ll}
0 &\hbox{if $i=j$;}\\
1 &\hbox{if $i\cdot j=0$;}\\
-u^{-a_{ij}}+v^{-a_{ji}} &\hbox{if $i\cdot j<0$ and $i\succ j$;}\\
u^{-a_{ij}}-v^{-a_{ji}} &\hbox{if $i\cdot j<0$ and $i\prec j$.}
\end{array}
\right.
\end{equation}

\subsection{\boldmath The definition and first properties}\label{SSDefKLR}
Fix a quiver $\Ga$ with a compatible automorphism and $\al\in Q_+$ of height $d$. Recall the polynomials $Q_{ij}(u,v)$ defined in (\ref{EQIJ}). 
Let $R_\al=R_\al(\Ga)$ be an associative graded unital $\FF$-algebra, given by the generators
\begin{equation}\label{EKLGens}
\{e(\bi)\mid \bi\in \words^\al\}\cup\{y_1,\dots,y_{d}\}\cup\{\psi_1, \dots,\psi_{d-1}\}
\end{equation}
and the following relations for all $\bi,\bj\in \words^\al$ and all admissible $r$ and $s$:
\begin{equation}
e(\bi) e(\bj) = \de_{\bi,\bj} e(\bi),
\quad{\textstyle\sum_{\bi \in \words^\alpha}} e(\bi) = 1;\label{R1}
\end{equation}
\begin{equation}\label{R2PsiY}
y_r e(\bi) = e(\bi) y_r;
\end{equation}
\begin{equation}
\psi_r e(\bi) = e(s_r\bi) \psi_r;\label{R2PsiE}
\end{equation}
\begin{equation}\label{R3Y}
y_r y_s = y_s y_r;
\end{equation}
\begin{equation}\label{R3YPsi}
y_r \psi_s = \psi_s y_r\qquad (r \neq s,s+1);
\end{equation}
\begin{equation}
(y_{r+1} \psi_r-\psi_r y_r) e(\bi) =
\left\{
\begin{array}{ll}
e(\bi) &\hbox{if $i_r=i_{r+1}$,}\\
0 &\hbox{if $i_r\neq i_{r+1}$;}
\end{array}
\right.
\label{R5}
\end{equation}
\begin{equation}
(y_r\psi_r-\psi_r y_{r+1})e(\bi) 
= 
\left\{
\begin{array}{ll}
-e(\bi) &\hbox{if $i_r=i_{r+1}$,}\\
0 &\hbox{if $i_r\neq i_{r+1}$;}
\end{array}
\right.
\label{R6}
\end{equation}
\begin{equation}
\psi_r^2e(\bi) = Q_{i_r,i_{r+1}}(y_r,y_{r+1})e(\bi)
 \label{R4}
\end{equation}
\begin{equation} 
\psi_r \psi_s = \psi_s \psi_r\qquad (|r-s|>1);\label{R3Psi}
\end{equation}
\begin{equation}
\begin{split}
&(\psi_{r+1}\psi_{r} \psi_{r+1}-\psi_{r} \psi_{r+1} \psi_{r}) e(\bi) 
\\=
&\left\{\begin{array}{ll}
\frac{Q_{i_r,i_{r+1}}(y_{r+2},y_{r+1})-Q_{i_r,i_{r+1}}(y_r,y_{r+1})}{y_{r+2}-y_r}e(\bi)&\text{if $i_r=i_{r+2}$,}\\
0 &\text{otherwise.}
\end{array}\right.
\end{split}
\label{R7}
\end{equation}
The {\em grading} on $R_\al$ is defined by setting:
$$
\deg(e(\bi))=0,\quad \deg(y_re(\bi))=i_r\cdot i_r,\quad\deg(\psi_r e(\bi))=-i_r\cdot i_{r+1}.
$$

\vspace{2 mm}

The expression in the right hand side of (\ref{R7}) should be interpreted as the substitution $u=y_r,v=y_{r+1},w=y_{r+2}$ into $\frac{Q_{i_r,i_{r+1}}(w,v)-Q_{i_r,i_{r+1}}(u,v)}{w-u}$, which is first interpreted as a {\em polynomial} in $\FF[u,v,w]$. 
It is pointed out in \cite{KL2} and \cite[\S3.2.4]{R} that up to isomorphism the graded $\FF$-algebra $R_\al$ depends only on the Cartan datum and $\al$. 

The algebra $R_\alpha$ possesses a graded anti-automorphism
\begin{equation}\label{star}
\tau:R_\alpha \rightarrow R_\alpha,
\end{equation}
which is the identity on generators. 

%\subsection{Basis Theorem}
For each element $w\in S_d$ fix a reduced expression $w=s_{r_1}\dots s_{r_m}$ and set 
$$
\psi_w:=\psi_{r_1}\dots \psi_{r_m}.
$$
In general, $\psi_w$ depends on the choice of the reduced expression of $w$. 

\begin{Theorem}\label{TBasis}{\cite[Theorem 2.5]{KL1}}, \cite[Theorem 3.7]{R} 
%{\bf (Basis Theorem)}
The elements 
%\begin{equation}\label{EBasis}
$$ \{\psi_w y_1^{m_1}\dots y_d^{m_d}e(\bi)\mid w\in S_d,\ m_1,\dots,m_d\in\Z_{\geq 0}, \ \bi\in \words^\al\}
$$ 
%\end{equation}
form an $\FF$-basis of  $R_\al$. 
\end{Theorem}

It is convenient to consider the direct sum of algebras 
%\begin{equation}\label{EOplus}
%R_d:=\bigoplus_{\al\in Q_+,\ \height(\al)=d} R_\al\qquad\text{and}\qquad
$
R:=\bigoplus_{\al_\in Q_+} R_\al.
$
%\end{equation}
We refer %to the algebra $R_d$ as the {\em (affine) Khovanov-Lauda-Rouquier algebra of rank $d$} and 
to the algebra $R$ as  the {\em (affine) Khovanov-Lauda-Rouquier algebra}. 
Note that $R$ is non-unital, but it is locally unital since each $R_\al$ is unital. It is shown in \cite[Corollary 2.11]{KL1} that the algebras $R_\al$ are indecomposable---in other words they are the blocks of $R$.

%Denote by $P_\al$ the (commutative) subalgebra of $R_\al$ generated by $y_1,\dots,y_d$ and all $\{e(\bi)\mid \bi\in \words^\al\}$. By the Basis Theorem, 
%$$\{y_1^{m_1}\dots y_d^{m_d}e(\bi)\mid m_1,\dots, m_d\in\Z_{\geq 0},\ \bi\in \words^\al\}$$ is a basis of $P_\al$. 

%In this section we begin to discuss representation theory of the quiver nil Hecke algebras $R$. The main result is the Khovanov-Lauda Theorem connecting projective modules over $R$ and the corresponding quantized enveloping algebra $\mathbf f$.

\subsection{\boldmath Basic representation theory of $R_\al$}
When considering representation theory of $R_\al$ we follow the notation of \cite{BKAdv}.
 
In this paper {\em grading} always means {\em $\Z$-grading}, unless otherwise stated. Recall that $R_\al$ is a graded algebra.
We are interested in graded representations of $R_\al$.  Let us first record some basic facts of graded representation theory.

For any graded $\FF$-algebra $H$ we denote by $\Mod{H}$ the abelian category of all graded left $H$-modules, with 
morphisms being {\em degree-preserving} module homomorphisms, which we denote by $\Hom$.
Let $\Rep{H}$ denote
the abelian subcategory of all
{\em finite dimensional}\, graded $H$-modules and
 $\Proj{H}$ denote the additive subcategory  of 
all {\em finitely generated projective}\, graded $H$-modules. 
Denote the corresponding Grothendieck groups by
 $[\Rep{H}]$ and $[\Proj{H}]$, respectively. Recall that $\Laurent=\Z[q,q^{-1}]$. Now, our  
Grothendieck groups are $\Laurent$-modules via 
%\begin{equation}\label{EAMod}
$
q^m[M]:=[M\langle m\rangle],
$ 
%\end{equation}
where $M\langle m\rangle$ denotes the module obtained by 
shifting the grading up by $m$:
\begin{equation}\label{obvious}
M\langle m\rangle_n:=M_{n-m}.
\end{equation}
%Given $f = \sum_{n \in \Z} f_n q^n \in \Z[[q,q^{-1}]]$ and $M\in\Mod{H}$,we write $$f \cdot M:=\bigoplus_{n \in \Z} M \langle n \rangle^{\oplus f_n}.$$

For $n \in \Z$, let
$
\Hom_H(M, N)_n := \Hom_H(M \langle n \rangle, N)
= \Hom_H(M, N \langle -n \rangle)
$
denote the space of all homomorphisms
that are homogeneous of degree $n$,
i.e. they map $M_i$ into $N_{i+n}$ for each $i \in \Z$.
Set
$$
\HOM_H(M,N) := \bigoplus_{n \in \Z} \Hom_H(M,N)_n. 
%\quad \END_H(M) := \HOM_H(M,M).
$$
For a finite dimensional 
graded vector space $V=\oplus_{n\in \Z} V_n$, its {\em graded dimension} is $\qdim \, V:=\sum_{n \in \Z}  (\dim V_n)q^n\in\Laurent$. 
There is a {\em natural pairing} 
$$\langle.,.\rangle:[\Proj{H}]\times[\Rep{H}] \rightarrow \Laurent,\quad \langle[P],[M]\rangle := \qdim\ \HOM_H(P,M),$$
%where $\qdim \, V$ denotes $\sum_{n \in \Z}  (\dim V_n)q^n\in\Laurent$ for any finite dimensional graded vector space $V=\oplus_{n\in \Z} V_n$.
%Note that the Cartan pairing is {\em sesquilinear}, i.e. anti-linear in the first argument and  linear in the second.

Given $M, L \in \Rep{H}$ with $L$ irreducible, 
we write $[M:L]_q$ for the corresponding {\em  graded composition multiplicity},
i.e. 
\begin{equation}\label{EGCM}
[M:L]_q := \sum_{n \in \Z} a_n q^n,
\end{equation}
where $a_n$ is the multiplicity
of $L\langle n\rangle$ in a graded composition series of $M$.

Going back to the algebras $R_\al$, every irreducible graded $R_\al$-module is finite dimensional \cite[Proposition 2.12]{KL1}, and  there are finitely many irreducible modules in $\Rep{R_\al}$ up to isomorphism and grading shift \cite[\S 2.5]{KL1}. For every irreducible module $L\in \Rep{R_\al}$ its {\em projective cover} $P_L$ is the (unique up to a homogeneous isomorphism) 
module in $\Proj{R_\al}$ such that for any irreducible $S\in \Rep{R_\al}$, we have 
$$
\qdim\HOM_{R_\al}(P_L,S)=
\left\{
\begin{array}{ll}
q^{-n} &\hbox{if $S\cong L\langle n\rangle$ for some $n\in \Z$;}\\
0 &\hbox{otherwise.}
\end{array}
\right. 
$$
Any object in $\Proj{R_\al}$ is a finite direct sum of indecomposable modules of the form $P_L\langle n\rangle$. 
%Since all $y_re(\bi)$ are positively graded, the elements $y_r$ act nilpotently on modules $M\in\Rep{R_\al}$. 

For $\bi\in \words^\al$ and $M\in\Rep{R_\al}$, the {\em $\bi$-weight space} of $M$ is
$
M_\bi:=e(\bi)M.
$
We have a decomposition of (graded) vector spaces
$
M=\bigoplus_{\bi\in \words^\al}M_\bi.
$
We say that $\bi$ is a {\em weight of $M$} if $M_\bi\neq 0$.
%, and refer to $\words^\al$, as the set of {\em weights for $R_\al$}. 
Note from the relations that 
%\begin{equation}\label{EAction}
$
\psi_r M_\bi\subset M_{s_r \bi}.
$
%\end{equation}

%Recall from (\ref{EOplus}) the algebra $R=\oplus_{\al\in Q_+}R_\al$. 
We identify in a natural way: 
$$
[\Rep{R}]=\bigoplus_{\al\in Q_+}[\Rep{R_\al}],\quad [\Proj{R}]=\bigoplus_{\al\in Q_+}[\Proj{R_\al}].
$$
Given $\alpha, \beta \in Q_+$, we define the (graded) algebra 
 $
R_{\alpha,\beta} := R_\alpha \otimes 
R_\beta.$ 
Denote 
the (outer) tensor product of an $R_\alpha$-module $M$ and an $R_\beta$-module 
$N$ by $M \boxtimes N$. We identify the Grothendieck group $[\Proj{R_{\alpha,\beta}}]$ with
$[\Proj{R_\alpha}] \otimes [\Proj{R_\beta}]$
so that $[P \boxtimes Q]$ is identified with $[P] \otimes [Q]$. Similarly we identify $[\Rep{R_{\alpha,\beta}}]$ with
$[\Rep{R_\alpha}] \otimes [\Rep{R_\beta}]$.

%If $V=\oplus_{k\in \Z} V[k]$ is a graded vector space, its {\em graded dimension} is $$\gdim V:=\sum_{k\in \Z}(\dim V[k])q^k\in\Z[q,q^{-1}]. $$ 
%Let $\Laurent[\words^\al]$ be the free $\Laurent$-module with basis $\{\bi\mid \bi\in \words^\al\}$. 

\subsection{Pairings and dualities}\label{SSPD}
Recall the map $\tau$ from (\ref{star}). There is a duality  $\#$
on $\Proj{R_\alpha}$ mapping $P\in\Proj{R_\al}$
to $P^\# := \HOM_{R_\alpha}(P, R_\alpha)$ with the action %defined by 
$(xf)(p) = f(p) \tau(x)$ for $x\in R_\al,f\in \HOM_{R_\alpha}(P, R_\alpha), p\in P$. 
There is also a duality $\circledast$
on $\Rep{R_\alpha}$ mapping  $M\in\Rep{R_\al}$
to $M^\circledast := \HOM_\FF(M, \FF)$ with the action  $(xf)(m) = f(\tau(x)m)$ for $x\in R_\al,f\in \HOM_\FF(M, \FF), m\in M$.

Let $M$ be a left $R_\al$-module. Denote by $M^\tau$ the right $R_\al$-module with the action given by $mx=\tau(x)m$ for $x\in R_\al,m\in M$. Following \cite[(2.44)]{KL1}, define the {\em Khovanov-Lauda pairing} to be the $\Laurent$-linear pairing such that 
$$(.,.) : [\Proj{R_\al}]\times[\Rep{R_\al}] \rightarrow \Laurent,\quad ([P],[M]) := \qdim(P^\tau\otimes_{R_\al} M).$$ The relation between the natural pairing $\langle\cdot,\cdot\rangle$, the Khovanov-Lauda pairing $(\cdot,\cdot)$, and the dualities is as follows (cf. \cite[p.341]{KL1}):

\begin{Lemma}\label{LFirst}%{\rm \cite{}}%{\bf ()}
Let $P\in\Proj{R_\al}$ and $M\in \Rep{R_\al}$. Then
$$
\overline{([P],[M^\circledast])}=\langle[P],[M]\rangle=([P^\#],[M]).%=\overline{\langle [P^\#],[M^\circledast]\rangle} .
$$ 
\end{Lemma}
\begin{proof}
Indeed, we have
\begin{align*}
\overline{([P],[M])}& = \qdim \HOM_\FF(P^\tau\otimes_{R_\al} M,\FF)
%\\
%&=\qdim \HOM_{R_\al}(P^\tau,\Hom_\FF(M,\FF))
\\
&=\qdim \HOM_{R_\al}(P,M^\circledast)=\langle[P],[M^\circledast]\rangle,
\end{align*}
which implies the first equality. Moreover, 
\begin{align*}
([P^\#],[M])& = \qdim (P^\#)^\tau\otimes_{R_\al} M
\\
&=\qdim \Hom_{R_\al}(P,R_\al)\otimes_{R_\al} M
\\
&=\qdim \Hom_{R_\al}(P, M)=\langle[P],[M]\rangle,
\end{align*}
which gives the second equality. 
\end{proof}

%The iso-classes of irreducible graded $R_\al$-modules and the classes of the corresponding  projective covers form a pair of dual bases in $\Rep{R_\al}$ and $\Proj{R_\al}$ with respect to the Cartan pairing. 

\begin{Lemma}\label{LThird}%{\rm \cite{}}%{\bf ()}
Let $L\in \Rep{R_\al}$ be irreducible. Then $P_L^\#\cong P_{L^\circledast}$.  
\end{Lemma}
\begin{proof}
By Lemma~\ref{LFirst}, we have 
$\langle P^\#,M\rangle=\overline{\langle P,M^\circledast\rangle}$ for any $P\in\Proj{R_\al}$ and $M\in\Rep{R_\al}$, which implies the result.
\end{proof}

It %is a general fact  that a grading on an irreducible module is unique up to a grading shift. Moreover, it 
is pointed out in \cite[p. 342]{KL1} that for every irreducible graded $R_\al$-module there is a (unique)  choice of the grading shift which makes the module $\circledast$-self-dual:

\begin{Lemma}\label{LSelfD} \cite{KL1} 
For every irreducible graded $R_\al$-module $L$ there is a unique $m\in\Z$ such that  $(L\langle m\rangle)^\circledast\cong L\langle m\rangle$. 
\end{Lemma}

If $L$ is an irreducible graded module with $L^\circledast\cong L$, then by Lemma~\ref{LThird}, we have $P_L^\#\cong P_L$. The classes $\{[L]\}$ of the self-dual irreducible $R_\al$-modules form an $\Laurent$-basis of $[\Rep{R_\al}]$, while the classes of the corresponding projective covers $\{[P_L]\}$ form a dual basis of $\Rep{R_\al}$ with respect to both natural and Khovanov-Lauda pairings.

%This commutes with the functor $\theta_i$, i.e.
%\begin{equation}\label{snow}\# \circ \theta_i \cong \theta_i \circ \#:\Proj{R_\alpha} \rightarrow \Proj{R_{\alpha+i}}.\end{equation}

\subsection{Induction and restriction}
Let $\al,\be\in Q_+$. 
There is an obvious (non-unital) algebra embedding
of $R_{\alpha,\beta}$
into the $R_{\alpha+\beta}$
mapping $e(\bi) \otimes e(\bj)$ to $e(\bi\bj)$.
%,where $\bi\bj$ denotes the concatenation of the two sequences. It is nota {\em unital} algebra homomorphism. 
The image of the identity
element of $R_{\alpha,\beta}$ under this map is 
$$
e_{\alpha,\beta} = \sum_{\bi \in \words^\alpha,\bj \in \words^\beta} e(\bi\bj).
$$
%Let $\Ind_{\alpha,\beta}^{\alpha+\beta}$ and $\Res_{\alpha,\beta}^{\alpha+\beta}$ denote 
Consider the %corresponding induction and restriction 
functors 
\begin{align*}
\Ind_{\alpha,\beta}^{\alpha+\beta} &:= R_{\alpha+\beta} e_{\alpha,\beta}
\otimes_{R_{\alpha,\beta}} ?:\Mod{R_{\alpha,\beta}} \rightarrow \Mod{R_{\alpha+\beta}},\\
\Res_{\alpha,\beta}^{\alpha+\beta} &:= e_{\alpha,\beta} R_{\alpha+\beta}
\otimes_{R_{\alpha+\beta}} ?:\Mod{R_{\alpha+\beta}}\rightarrow \Mod{R_{\alpha,\beta}}.
\end{align*}
For $\al,\be\in Q_+$, $M\in\Rep{R_\al}$ and $N\in \Rep{R_\beta}$, we sometimes denote
\begin{equation}\label{ECircle}
M\circ N:=\Ind_{\al,\be}^{\al+\be}(M\boxtimes N).
\end{equation}

Note $\Res_{\alpha,\beta}^{\alpha+\beta}$ is just left multiplication by
the idempotent $e_{\alpha,\beta}$, so it is exact and sends finite dimensional modules to
finite dimensional modules. 
By \cite[Proposition 2.16]{KL1},
$e_{\alpha,\beta} R_{\alpha+\beta}$ is a graded free left $R_{\alpha,\beta}$-module of finite rank,
so $\Res_{\alpha,\beta}^{\alpha+\beta}$ also sends finitely generated projectives to finitely generated projectives.
Similarly, $R_{\alpha+\beta} e_{\alpha,\beta}$ is a graded 
free right $R_{\alpha, \beta}$-module of finite rank, so
$\Ind_{\alpha,\beta}^{\alpha+\beta}$ is exact and sends finite dimensional modules to finite dimensional modules.
The functor $\Ind_{\alpha,\beta}^{\alpha+\beta}$ is left adjoint to $\Res_{\alpha,\beta}^{\alpha+\beta}$, and it sends finitely generated projectives to finitely generated projectives.

Since the functors of induction are exact, by taking direct sum, they define products 
on the Grothendieck groups $[\Rep{R}]$ and $[\Proj{R}]$. For reasons, which will become clear in section~\ref{SSCat}, we use use different notations for these products on $[\Rep{R}]$ and $[\Proj{R}]$: if $M,N\in\Rep{R}$, we write $[M]\circ[N]=[M\circ N]$; if $M,N\in\Proj{R}$, we write $[M][N]=[M\circ N]$. 
%`$\circ$' on the Grothendieck groups $[\Rep{R}]$ and $[\Proj{R}]$ such that $[M]\circ[N]=[M\circ N]$, see \cite[Proposition 3.1]{KL1}. 
Similarly the functors of restriction define coproducts $\Delta$ and $r$  on $[\Rep{R}]$ and $[\Proj{R}]$, respectively. These products and coproducts make $[\Rep{R}]$ and $[\Proj{R}]$ into twisted unital and counital bialgebras \cite[Proposition 3.2]{KL1}. The next result, which has the same proof as \cite[Proposition 3.3]{KL1}, shows that these bialgebras are dual to each other with respect to the Khovanov-Lauda pairing. 

\begin{Lemma}\label{LProdCoprod}%{\rm \cite{}}%{\bf ()}
For $x,x'\in[\Proj{R}]$ and $y,y'\in[\Rep{R}]$, we have 
$$
(x,y\circ y')=(r(x),y\otimes y'),\quad (xx',y)=(x\otimes x',\Delta(y)).
$$
\end{Lemma}

\subsection{Divided power functors}
Let $i\in I$ and $n\in\Z_{> 0}$. 
As explained in \cite[$\S$2.2]{KL1}, 
%in the case $\alpha = n i$ for some $i \in I$,
the algebra
$R_{n i}$ 
%is isomorphic to the usual affine nil Hecke algebra. It 
has a representation on the polynomials $\FF[y_1,\dots,y_n]$
such that each $y_t$ acts as multiplication by $y_t$ and each $\psi_t$ acts
as the divided difference operator 
$
\partial_t: f \mapsto \frac{{^{s_t}} f - f}{y_{t}-y_{t+1}}.
$
Let $P(i^{(n)})$ denote this representation of $R_{ni}$
viewed as a 
graded $R_{ni}$-module with grading defined by
$$
\deg(y_1^{m_1} \cdots y_n^{m_n}) := (i\cdot i)(m_1+\cdots+m_n - n(n-1)/4).
$$
By \cite[$\S$2.2]{KL1}, the left regular $R_{ni}$-module decomposes as 
$
P(i^n) \cong [n]^!_i \cdot P(i^{(n)})$. In particular, $P(i^{(n)})$ is projective.
Set
\begin{align*}%\label{div1}
\theta_{i}^{(n)}&:= 
\Ind_{\alpha,n i}^{\alpha+ni} (? \boxtimes P(i^{(n)})):\Mod{R_\alpha} \rightarrow \Mod{R_{\alpha+ni}},\\
(\theta_{i}^*)^{(n)}&:= 
\HOM_{R'_{n i}}(P(i^{(n)}), ?): \Mod{R_{\alpha+ni}} \rightarrow \Mod{R_{\alpha}},%\label{div2}
\end{align*}
where $R'_{ni} := 1 \otimes R_{ni} \subseteq R_{\alpha,ni}$.
Both functors are exact, 
$\theta_i^{(n)}$ sends finitely generated projective modules to 
finitely generated projective modules,
and $(\theta_i^*)^{(n)}$ sends finite dimensional modules to finite dimensional modules.
So the functors induce $\Laurent$-module maps on the corresponding Grothendieck groups: 
$$\theta_i^{(n)}:[\Proj{R_\alpha}] \rightarrow [\Proj{R_{\alpha+ni}}],\quad
(\theta_{i}^*)^{(n)}:[\Rep{R_{\alpha+ni}}] \rightarrow [\Rep{R_{\alpha}}].
$$ 
By taking direct sums over all $\al\in Q_+$, we have $\Laurent$-module homomorphisms
$$\theta_i^{(n)}:[\Proj{R}] \rightarrow [\Proj{R}],\quad
(\theta_{i}^*)^{(n)}:[\Rep{R}] \rightarrow [\Rep{R}].
$$

\begin{Lemma}\label{LThetaTheta*}%{\rm \cite{}}%{\bf ()}
For $x\in [\Proj{R}]$, $y\in [\Rep{R}]$, $i\in I$ and $n\in \Z_{\geq 0}$, we have 
$(\theta_i^{(n)} x,y)=(x,(\theta_i^*)^{(n)} y).
$
\end{Lemma}
\begin{proof}
This follows easily from the definitions and Lemma~\ref{LProdCoprod}, if we take into account the fact that $P(i^{(n)})$ is the projective cover of the only irreducible $R_{ni}$-module, cf. \cite[\S 2.2(3) and Proposition 3.11]{KL1}. 
\end{proof}

\section{Quantum groups and their categorifications}

%\subsection{Fields and parameters}
%Throughout the paper, $q$ is an indeterminate, and $\Q(q)$ is the field of rational fractions. We denote $\Laurent:=\Z[q,q^{-1}]$. 

\subsection{\boldmath The algebras $\mathbf f$ and $\mathbf{'f}$}
Let $\mathbf f$ and $\mathbf{'f}$ denote the Lusztig's algebras from \cite[$\S$1.2]{Lubook}
attached to the Cartan datum $(I,\cdot)$ over the field $\Q(q)$. We adopt the conventions of \cite[$\S$3.1]{KL1}, so our $q$ is Lusztig's $v^{-1}$.
To be more precise, $\mathbf{'f}$ is the free $\Q(q)$-algebra with generators $\theta_i'$ for $i\in I$. A $Q_+$-grading $\mathbf{'f}=\oplus_{\al\in Q_+} \mathbf{'f}_\al$ is determined by assigning the degree $i$ to the generator $\theta_i'$ for each $i\in I$. If $x\in \mathbf{'f}_\al$ we write $|x|=\al$. 

We consider $\mathbf{'f}\otimes \mathbf{'f}$  as a $\Q(q)$-algebra with {\em twisted}\, multiplication 
$(x\otimes y)(z\otimes w)=q^{-|y|\cdot |z|}xz\otimes yw$ for homogeneous $x,y,z,w\in \mathbf{'f}$. 
Let $r:\mathbf{'f}\to \mathbf{'f}\otimes\mathbf{'f}$ be an algebra homomorphism determined by $r(\theta_i')=\theta_i'\otimes 1+1\otimes \theta_i'$ for all $i\in I$. It is proved in \cite[1.2.2]{Lubook} that $r$ is coassociative. The twisted multiplication and the coproduct $r$ make $\mathbf{'f}$ into a {\em twisted bialgebra}. 
Let\,\, %$\bar{ }:\Q(q)\to \Q(q)$ be the $\Q$-algebra involution with $\bar q=q^{-1}$, and 
%A map $\phi:V\to W$ between $\Q(q)$-vector spaces is {\em antilinear} if $\phi(v+v')=\phi(v)+\phi(v')$ and $\phi(pv)=\bar p\phi(v)$ for all $v,v'\in V,p\in \Q(q)$. 
$\bar{ }\,:\mathbf{'f}\to \mathbf{'f}$ be the $\Q$-algebra homomorphism such that
 $\overline{p\theta_i'}=\bar p\theta_i'$ for all $p\in \Q(q)$ and $i\in I$. 
 
 Recall that $\Laurent=\Z[q,q^{-1}]$. Denote by 
 $\mathbf{'f_\Laurent}$ the $\Laurent$-subalgebra of $\mathbf{'f}$ generated by $\{(\theta_i')^{(n)}\mid i\in I, n\in \Z_{\geq 0}\}$. We have $\mathbf{'f_\Laurent}=\oplus_{\al\in Q_+}(\mathbf{'f}_\Laurent)_\al$, where $(\mathbf{'f}_\Laurent)_\al:=\mathbf{'f}_\Laurent\cap\mathbf{'f}_\al$. On restriction, we have well defined maps $r:\mathbf{'f_{\Laurent}}\to\mathbf{'f_{\Laurent}}\otimes\mathbf{'f_{\Laurent}}$ and\,\, $\bar{ }\,:\mathbf{'f}_{\Laurent}\to\mathbf{'f}_{\Laurent}$. Given $\bi=(j_1)^{r_1}\dots (j_m)^{r_m}\in \words$, with $j_n\neq j_{n+1}$ for $n=1,\dots,m-1$, denote  %written in the reduced form, denote 
 $
 \theta'_\bi:=(\theta'_{j_1})^{(r_1)}\dots(\theta'_{j_m})^{(r_m)}\in \mathbf{'f_{\Laurent}}.
 $ 
 Then $\{\theta'_\bi\mid \bi\in\words\}$ is an $\Laurent$-basis of $\mathbf{'f_{\Laurent}}$. 
 
Let $\mathbf f:=\mathbf{'f}/\mathcal{I}$, where $\mathcal{I}$ is the ideal of $\mathbf{'f}$ generated by the elements  
%is the $\Q(q)$-algebra on generators $\{\theta_i\mid i \in I\}$ subject to
%the quantum Serre relations
\begin{equation}\label{qserre}
%(\ad_q \theta_i)^{1-a_{j,i}}(\theta_j)=0
\sum_{m=0}^{1-a_{ij}}(-1)^m
%\left[
%\begin{matrix}
% 1-a_{ij}  \\
% m
%\end{matrix}
%\right]_i
(\theta_i')^{(m)}\theta_j'(\theta_i')^{(1-a_{ij}-m)}
\end{equation}
for all $i\neq j$ in $I$. Set $\theta_i:=\theta_i'+\mathcal{I}\in {\mathbf f}$ for all $i\in I$.  The $Q_+$-grading on $\mathbf{'f}$ yields the $Q_+$ grading on $\mathbf{f}$ with $|\theta_i|=i$. The maps $r$ and\,\, $\bar{}$\,\, yield maps $r:\mathbf{f}\to\mathbf{f}\otimes\mathbf{f}$ and\,\, $\bar{}\,:\mathbf{f}\to\mathbf{f}$. If $x\in\mathbf f_\gamma$, we can write
$$
r(x)=\sum_{\al,\be\in Q_+,\ \al+\be=\gamma}r_{\al,\be}(x)
$$
for $r_{\al,\be}(x)\in\mathbf f_\al\otimes \mathbf f_\beta$.

Let $\mathbf{f}_{\Laurent}$
be the $\Laurent$-subalgebra of $\mathbf f$ generated by all %the divided powers
$\theta_i^{(n)}$. %:= \theta_i^n / [n]^!_i$. 
The map $r$ restricts to a well-defined comultiplication
$r:\mathbf{f}_{\Laurent} \rightarrow \mathbf{f}_{\Laurent} \otimes \mathbf{f}_{\Laurent}$, and the bar-involution induces an involution $ \bar{ }\,:\mathbf{f}_{\Laurent}\to\mathbf{f}_{\Laurent}$.
Finally, $\mathbf{f}_{\Laurent}$ is a $Q_+$-graded $\Laurent$-algebra with $(\mathbf{f}_{\Laurent})_\al=\mathbf{f}_{\Laurent}\cap \mathbf{f}_\al$.

\subsection{\boldmath The algebras $\mathbf{'f}^*$ and $\mathbf{f}^*$}
Consider the graded duals 
$\mathbf{'f}^*:=\oplus_{\al\in Q_+}(\mathbf{'f})_\al^*$ and  $\mathbf{f}^*:=\oplus_{\al\in Q_+}(\mathbf{f})_\al^*.
$ 
Let $\iota: \mathbf{f}^*\into \mathbf{'f}^*$ be the map dual to the natural  quotient map $\mathbf{'f}\onto \mathbf{f}$. 
We interpret words $\bi \in\words$ as elements of $\mathbf{'f}^*$ so that 
$
\bi(\theta'_{j_1}\dots\theta'_{j_d})=\de_{\bi,\bj}, 
$ 
for $\bj=(j_1,\dots,j_d)$. Then $\words$ is a basis of $\mathbf{'f}^*$. 

The dual map $r^*:\mathbf{'f}^*\otimes \mathbf{'f}^*\to \mathbf{'f}^*$ is an associative product on $\mathbf{'f}^*$ which is denoted $\circ$ and called the {\em quantum shuffle product}. 
Let $\bi=(i_1,\dots,i_d)$ and $\bj=(i_{d+1},\dots,i_{d+f})$ be two elements of $\words$. 
 It is easy to check that
\begin{equation}\label{EQSh}
\bi\circ\bj:=\sum q^{-e(\sigma)}(i_{\sigma(1)},\dots,i_{\sigma(d+f)}),
\end{equation}
where the sum is over all $\sigma\in S_{d+f}$ such that $\sigma^{-1}(1)<\dots<\sigma^{-1}(d)$ and $\sigma^{-1}(d+1)<\dots<\sigma^{-1}(d+f)$ (i.e. minimal length coset representatives of $S_d\times S_f$ in $S_{d+f}$), and 
\begin{equation*}%\label{}
e(\sigma):=\sum_{k\leq d<m,\ \sigma^{-1}(k)>\sigma^{-1}(m)} i_{\si(k)}\cdot i_{\sigma(m)}.
\end{equation*}
For example, in type $A_2$ we have $(1)\circ (2)=(1,2)+q(2,1)$ and $(1)\circ (1)=(1+q^{-2})(1,1)$.

The dual map to the product map $m:\mathbf{'f}\otimes \mathbf{'f}\to \mathbf{'f}$ is the map
$$
\Delta:\mathbf{'f}^*\to \mathbf{'f}^*\otimes \mathbf{'f}^*,\ \bi\mapsto \sum_{\bj,\bk\in\words,\ \bi=\bj\bk}\bj\otimes \bk \qquad(\bi\in\words).
$$
In this way $\mathbf{'f}^*$ becomes a twisted bialgebra with twisted product 
 on $\mathbf{'f}^*\otimes \mathbf{'f}^*$ given by $(x\otimes y)\circ(z\otimes w)=q^{-|y|\cdot|z|}(x\circ z)\otimes (y\circ w)$ for homogeneous $x,y,z,w\in  \mathbf{'f}^*$. 
The dual map to the map\,\, $\bar{ }:\,  \mathbf{'f}\to \mathbf{'f}$ is the map\,\, $\bar{ }\,:  \mathbf{'f}^*\to \mathbf{'f}^*$, which satisfies
$\overline{p\bi}=\bar p\bi$ for all $p\in\Q(q)$ and $\bi\in\words$.

Note that $\mathbf{f}^*$ is a bar-invariant subalgebra of $\mathbf{'f}^*$ (with respect to the $\circ $-product), and $\De(\mathbf{f}^*)\subset \mathbf{f}^*\otimes\mathbf{f}^*$. If $x\in {\mathbf f}^*_\gamma$ for some $\ga\in Q_+$, we can write
$$
\De(x)=\sum_{\al,\be\in Q_+,\ \al+\be=\gamma}\De_{\al,\be}(x)\qquad\qquad\qquad(\De_{\al,\be}(x)\in\mathbf{f}^*_\al\otimes \mathbf{f}^*_\beta).
$$
%for $\De_{\al,\be}(x)\in\mathbf{f}^*_\al\otimes \mathbf{f}^*_\beta$. 

Now pass to $\Laurent$-forms. Let 
$\mathbf{'f}_\Laurent^*:=\{x\in\mathbf{'f}^*\mid x(\mathbf{'f}_\Laurent)\subset\Laurent\}$ and $\mathbf{f}_\Laurent^*:=\{x\in\mathbf{f}^*\mid x(\mathbf{f}_\Laurent)\subset\Laurent\}.$ 
As an $\Laurent$-module, $\mathbf{'f}_\Laurent^*$ is free with basis 
$\{[\bi]!\,\bi\mid\bi\in \words\}.$ 
The embedding $\iota:\mathbf{f}^*\hookrightarrow \mathbf{'f}^*$ restricts to the embedding 
%\begin{equation}\label{EIota}
$
\iota:\mathbf{f}_\Laurent^*\hookrightarrow \mathbf{'f}_\Laurent^*
$, 
%\end{equation}
which we always use to identify $\mathbf{f}_\Laurent^*$ with an $\Laurent$-submodule of $\mathbf{'f}_\Laurent^*$. The quantum shuffle product, the coproduct $\De$, and the bar-involution restrict to define the corresponding notions on $\mathbf{'f}_\Laurent^*$ and $\mathbf{f}_\Laurent^*$. 
%We note that i

\iffalse{
In general, $\mathbf{f}_\Laurent^*$ is larger than the $\Laurent$-subalgebra of 
$\mathbf{'f}_\Laurent^*$ generated by all $i\in I$. %For example, if $i\cdot j=-1$, then $(i,j)\in \mathbf{f}_\Laurent^*$, but 
However, on extending scalars to $\Q(q)$, they coincide. In fact, more is true:

\begin{Lemma}\label{LOnto}%{\rm \cite{}}%{\bf ()}
There is an isomorphism of twisted bialgebras $\mathbf f\iso \mathbf{f}_\Laurent^*\otimes_\Laurent\Q(q)$ which maps $\theta_i\mapsto i\otimes 1$ %\frac{1}{1-q^{i\cdot i}}$
 for all $i\in I$. In particular,
 \begin{equation}\label{Eqserre}
%(\ad_q \theta_i)^{1-a_{j,i}}(\theta_j)=0
\sum_{m=0}^{1-a_{ij}}(-1)^m\left[
\begin{matrix}
 1-a_{ij}  \\
 m
\end{matrix}
\right]_i i^{\circ m}\circ j\circ  i^{\circ (1-a_{ij}-m)}=0.
\end{equation}
\end{Lemma}
\begin{proof}
This follows from \cite[1.2.3,1.2.4,1.2.5]{Lubook}.
\end{proof}
}\fi

Let $i\in I$. Denote by 
$
\theta_i^*:\mathbf{'f}^*\to \mathbf{'f}^*
$ the dual map to the map $\mathbf{f}\to \mathbf{f},\ x\mapsto x\theta_i'.$
For $\bj=(j_1,\dots,j_d)\in\words$, we have
\begin{equation*}\label{ECut}
\theta_i^*(\bj)=
\left\{
\begin{array}{ll}
(j_1,\dots,j_{d-1}) &\hbox{if $j_d=i$;}\\
0 &\hbox{otherwise.}
\end{array}
\right.
\end{equation*}
It is clear that $\theta_i^*$ preserves $\mathbf{'f}_\Laurent^*$ and $\mathbf{f}_\Laurent^*$. Moreover, the divided power $(\theta_i^*)^{(n)}$ of $\theta_i^*$ is dual to the map $\mathbf{f}_\Laurent\to \mathbf{f}_\Laurent,\ x\mapsto x\theta_i^{(n)}$, and so it also preserves $\mathbf{f}_\Laurent^*$.

\begin{Lemma}\label{LQS}%{\rm \cite{}}%{\bf ()}
Let $i\neq j\in I$. 
\begin{enumerate}
\item[{\rm (i)}] The map $\theta_i^*:\mathbf{'f}^*\to \mathbf{'f}^*$ is a $q$-derivation with respect to the shuffle product in the following sense: $\theta_i^*(\bj\circ \bk)=q^{-(i\cdot|\bk|)}\theta_i^*(\bj)\circ \bk+\bj\circ \theta_i^*(\bk)$ for all  $\bj,\bk\in\words$. 
\item[{\rm (ii)}]  On restriction to ${\bf f}^*_\Laurent$, the maps $\theta_i^*$ and $\theta_j^*$ satisfy quantum Serre relations, i.e.
$
\sum_{m=0}^{1-a_{ij}}(-1)^m
%\left[
%\begin{matrix}
% 1-a_{ij}  \\
% m
%\end{matrix}
%\right]_i
(\theta_i^*)^{(m)}\theta_j^*(\theta_i^*)^{(1-a_{ij}-m)}=0.
$
\end{enumerate}
\end{Lemma}
\begin{proof}
Follows from definitions, cf. \cite[(4) and Lemma 3]{Lec}. 
\end{proof}

\subsection{\boldmath Categorification of $\mathbf f_\Laurent$ and $\mathbf f_\Laurent^*$}\label{SSCat}
Now we can state the fundamental categorification theorem 
proved in \cite[\S 3]{KL1}.

\begin{Theorem}[Khovanov-Lauda]\label{klthm}
There is a unique $\Laurent$-linear isomorphism
$
\gamma:\mathbf f_\Laurent \stackrel{\sim}{\rightarrow} [\Proj{R}]
$
such that 
$1 \mapsto [R_0]$ (the class of the left
regular representation of
the trivial algebra $R_0\cong \FF$) 
and 
$\gamma( x \theta_i^{(n)}) = 
\theta_i^{(n)}(\gamma(x))$ 
for all $x \in \mathbf f_\Laurent$,  $i \in I$, and $n \geq 1$. 
Under this isomorphism:
\begin{itemize}
\item[(1)] $\gamma\big((\mathbf f_\Laurent)_\al\big)=[\Proj{R_\al}]$;
\item[(2)] the 
multiplication $(\mathbf f_\Laurent)_\alpha \otimes (\mathbf f_\Laurent)_\beta
\rightarrow  (\mathbf f_\Laurent)_{\alpha+\beta}$ 
corresponds to the product on $[\Proj{R}]$ induced by the exact functor 
$\Ind_{\alpha,\beta}^{\alpha+\beta}$;
\item[(3)] the comultiplication $r_{\alpha,\beta}:(\mathbf f_\Laurent)_{\alpha+\beta} 
\rightarrow (\mathbf f_\Laurent)_\alpha \otimes (\mathbf f_\Laurent)_\beta$
corresponds to the coproduct on $[\Proj{R}]$ induced by the 
exact functor
$\Res_{\alpha,\beta}^{\alpha+\beta}$;
\item[(4)] the bar-involution on $\mathbf f_{\Laurent}$ 
corresponds to the
anti-linear involution induced by the duality $\#$.
\end{itemize}
\end{Theorem}

%\subsection{\boldmath Categorification of $\mathbf f_\Laurent^*$}

Let $M$ be a finite dimensional graded $R_\al$-module. Define the {\em $q$-character} of $M$ as follows: 
\begin{equation*}%\label{grch}
\CH M:=\sum_{\bi\in \words^\al}(\qdim M_\bi) \bi\in \mathbf{'f}_\Laurent^*.
\end{equation*}
The $q$-character map $\CH: \Rep{R_\al}\to \mathbf{'f}_\Laurent^*$ factors through to give an $ \Laurent$-linear map from the Grothendieck group
\begin{equation}\label{EChMap}
\CH: [\Rep{R_\al}]\to  \mathbf{'f}_\Laurent^*.
\end{equation}

It easily follows from the definitions and \cite[Lemma 3.5]{KL1} that

\begin{Lemma}\label{LThEq}%{\rm \cite{}}%{\bf ()}
The map $\CH: [\Rep{R_\al}]\to  \mathbf{'f}_\Laurent^*$ is $\theta_i^*$-equivariant for all $i\in I$. 
\end{Lemma}

%The following result %from \cite[Lemma 2.20]{KL1} 
%connects induction to quantum shuffles: % (note change in the order of multiples):

%The following result shows that the $q$-characters of the irreducible (graded) $R_\al$-modules are linearly independent. 

%\begin{Theorem}\label{TFCh}{\rm \cite[Theorem~3.17]{KL1}}\,%{\bf ()}
%The map (\ref{EChMap}) is injective. 
%\end{Theorem}

We now state a dual result to Theorem~\ref{klthm}. 

\begin{Theorem}\label{Dklthm} 
There is an $\Laurent$-linear isomorphism
$
\gamma^*: [\Rep{R}]\stackrel{\sim}{\rightarrow}\mathbf f_\Laurent^*
$
with the following properties:
\begin{itemize}
\item[(1)] $\ga^*([R_0])= 1$, where $R_0$ is the left regular representation of $R_0$; 
\item[(2)] $\gamma^*((\theta_i^*)^{(n)}(x)) = 
(\theta_i^*)^{(n)}(\gamma^*(x))$ for all $x \in [\Rep{R}],\ 
i \in I,\ n \geq 1$; 
\item[(3)] the following triangle is commutative:
$$
\begin{pb-diagram}
\node{}\node{\mathbf{'f}_\Laurent^*}%\arrow{ne,t}{\iota}\arrow{nw,t}{\phi'}
\node{} \\
\node{[\Rep{R}]} \arrow[2]{e,t}{\ga^*}
\arrow{ne,t}{\CH}
\node{}\node{\mathbf{f}_\Laurent^*}
\arrow{nw,t}{\iota}
\end{pb-diagram}
$$
\item[(4)] $\gamma^*([\Rep{R_\al}])=(\mathbf f^*_\Laurent)_\al$ for all $\ga\in Q_+$;
\item[(5)] under the isomorphism $\ga^*$, the 
multiplication $(\mathbf f^*_\Laurent)_\alpha \otimes (\mathbf f^*_\Laurent)_\beta
\rightarrow  (\mathbf f^*_\Laurent)_{\alpha+\beta}$  
corresponds to the product on $[\Rep{R}]$ induced by %the exact functor 
$\Ind_{\alpha,\beta}^{\alpha+\beta}$;
\item[(6)] under the isomorphism $\ga^*$, the comultiplication $\Delta_{\alpha,\beta}:(\mathbf f^*_\Laurent)_{\alpha+\beta} 
\rightarrow (\mathbf f^*_\Laurent)_\alpha \otimes (\mathbf f^*_\Laurent)_\beta$
corresponds to the coproduct on $[\Rep{R}]$ induced by %the exact functor
$\Res_{\alpha,\beta}^{\alpha+\beta}$;
\item[(7)] under the isomorphism $\ga^*$, the bar-involution on $\mathbf f^*_{\Laurent}$ 
corresponds to the
anti-linear involution induced by the duality $\circledast$.
\end{itemize}
\end{Theorem}
\begin{proof}
As explained in section~\ref{SSPD}, there is a non-trivial bilinear pairing $(\cdot,\cdot):[\Proj{R}]\times[\Rep{R}]\to\Laurent$ defined by $([P],[M])=\qdim(P^\tau\otimes_{R} M)$. Using this pairing we identify $[\Rep{R}]$ with $[\Proj{R}]^*$ and define $\ga^*$ to be the dual map to the map $\ga$ from Theorem~\ref{klthm}. Then (1) and (4) are clear from Theorem~\ref{klthm}. Part (2) follows from Theorem~\ref{klthm} and Lemma~\ref{LThetaTheta*}. Part (7) follows from Theorem~\ref{klthm}(4) and Lemma~\ref{LThird}. 
Parts (5),(6) come from Theorem~\ref{klthm}(2),(3) and Lemma~\ref{LProdCoprod}. 

Finally for (3), we apply induction on the height of $\al$ to prove that $\iota\circ\ga^*(x)=\CH(x)$ for any $x\in [\Rep{R_\al}]$. The base of induction is clear. If $\height(\al)>0$, then it suffices to prove that $\theta_i^*(\iota\circ\ga^*(x)-\CH(x))=0$ for any $i\in I$. But the maps $\iota$, $\ga^*$ and $\ch_q$ are all $\theta_i^*$-equivariant, see Lemma~\ref{LThEq}. So we are reduced to checking that $\iota\circ\ga^*(\theta_i^*(x))-\CH(\theta_i^*(x))=0$, which is true by induction. 
\end{proof}

We point out some consequences, all of which have been noted  in \cite{KL1}. 

\begin{Corollary}\label{CInjective}%
{\rm \cite[Theorem 3.17]{KL1}} %{\bf ()}
The $q$-character map (\ref{EChMap}) is injective.
\end{Corollary}
\begin{proof}
Follows from Theorem~\ref{Dklthm}(3). 
\end{proof}

\begin{Corollary}\label{CChSh}%{\rm \cite{}}%{\bf ()}
\cite[Lemma 2.20]{KL1}
Let $\al,\be\in Q_+$, $M\in\Rep{R_\al}$, $N\in \Rep{R_\beta}$. Then
$
\CH(M\circ N)=(\CH M)\circ (\CH N).
$
\end{Corollary}
\begin{proof}
Follows from Theorem~\ref{Dklthm}(3),(5). 
\end{proof}

\iffalse{
\begin{Corollary} {\rm \cite[Corollary 2.15]{KL1}} 
Let $M\in\Rep{R_\al}$ and $i,j\in I$. Suppose that $\bi\in \words^\al$ satisfies $i_1=j$ and $i_r=i$ for all $2\leq r\leq 2-a_{i,j}$. %Denote $w_r:=s_rs_{r-1}\dots s_1$ for $r=0,1,\dots,
$$
\sum_{m=0}^{1-a_{ij}}(-1)^m\left[
\begin{matrix}
 1-a_{ij}  \\
 m
\end{matrix}
\right]_i\qdim M_{s_{m}s_{m-1}\dots s_1\bi}=0. 
$$
\end{Corollary}
\begin{proof}
Follows from Theorem~\ref{Dklthm}(2),(3) and Lemma~\ref{LQS}(ii). 
\end{proof}
}\fi

\section{Combinatorics of words and dual-canonical basis}
It was first noticed in \cite{LR} and further developed in \cite{Ro} and \cite{Lec} that the so-called good Lyndon words play an important role in describing various bases of $\mathbf f$. Since $R_\al$ is closely related to  $\mathbf f$, it is natural to expect that good Lyndon words should also be important for representation theory of $R_\al$. %This has been predicted by Leclerc \cite{Lec}. 
%In this section we collect the necessary facts from the above references which will be used later. 

\subsection{Lexicographic order and Lyndon words}
%We recall some notions from \cite{Lec}. Note that the quantum shuffle product used here   is the opposite to the one used in \cite{Lec}. 
%---to transfer the results and definitions of \cite{Lec} use the antiautomorphism $\tau$ from \cite[Proposition 6]{Lec}. 
Let us fix a total order `$<$' on the set $I$, and denote by the same symbol `$<$' the corresponding lexicographical order on the set of words $\words$. 
If $x\in\mathbf{'f}^*$ we denote by $\max(x)$ the largest word appearing in $x$. 

\begin{Lemma}\label{LLex}%{\rm \cite{}}%{\bf ()}
Let $m\in\Z_{>0}$, and $\bi^{(r)},\bj^{(r)}\in I^{d_r}$ with $\bi^{(r)}\leq\bj^{(r)}$ for $r=1,2,\dots,m$. Then $\max(\bi^{(1)}\circ \bi^{(2)}\circ \dots\circ \bi^{(r)})\leq\max(\bj^{(1)}\circ \bj^{(2)}\circ \dots\circ \bj^{(r)})$. If in addition the strict inequality $\bi^{(r)}<\bj^{(r)}$ holds for some $1\leq r\leq m$, then $\max(\bi^{(1)}\circ \bi^{(2)}\circ \dots\circ \bi^{(r)})<\max(\bj^{(1)}\circ \bj^{(2)}\circ \dots\circ \bj^{(r)})$. 
\end{Lemma}
\begin{proof}
Any shuffle of the words $\bi^{(1)},\bi^{(2)},\dots,\bi^{(r)}$ is less than or equal to the same shuffle applied to the words $\bj^{(1)},\bj^{(2)},\dots,\bj^{(r)}$, and the strict inequality holds if the second stronger assumption of the lemma holds. 
%The result follows. 
\end{proof}

%\subsection{Lyndon words}
%In order to describe good words we introduce  Lyndon words. 
A word $\bi\neq\emptyset$ is called {\em Lyndon} if it is lexicographically smaller than all its proper right factors. It is well known \cite{Lo,Reu} that every word $\bi$ has a unique factorization  $\bi=\bi^{(1)}\bi^{(2)}\cdots\bi^{(k)}$ such that $\bi^{(1)}\geq \bi^{(2)}\geq \dots\geq \bi^{(k)}$ are Lyndon words. We refer to this as the {\em canonical factorization} of $\bi$. 

%The factorization of a good word $\bi$ as a weakly deceasing product of good Lyndon words as in Lemma~\ref{LLR1} will be referred to as the {\em canonical factorization of $\bi$}. It is known to be unique \cite{Lo,Reu}. 

\begin{Lemma}\label{LLL}%{\rm \cite{}}%{\bf ()}
Let $\bi,\bj,\bk,\bl\in\words$, $\bl\geq \bi$, $\bl=\bj\bk$, $\bi,\bk\neq\emptyset$, and assume that $\bl$ is Lyndon. Then $\bj\bi\bk<\bl\bi$ unless $\bl=\bk=\bi$. 
\end{Lemma}
\begin{proof}
Let $\bi=(i_1,\dots,i_b)$, $\bl=(l_1,\dots,l_a)$. 
By \cite[Lemma 15]{Lec}, we have $\max(\bl\circ \bi)=\bl\bi$. As $\bj\bi\bk$ is a shuffle of $\bl$ and $\bi$, we may assume that $\bj\bi\bk=\bl\bi$. It follows that $(i_1,\dots,i_c)=(l_d,\dots,l_b)$ for some $1\leq c\leq b$ and $1\leq d\leq b$. Moreover, assume that the condition $\bl=\bk=\bi$ does not hold. Then $c<b$ or $d>1$. Now, as $\bl$ is Lyndon, we have $(l_d,\dots,l_b)\geq \bl\geq\bi\geq (i_1,\dots,i_c)$. So $(i_1,\dots,i_c)=(l_d,\dots,l_b)$ implies $(l_d,\dots,l_b)= \bl=\bi= (i_1,\dots,i_c)$. But $c<b$ implies $(i_1,\dots,i_c)<\bi$, and  $d>1$ implies $(l_d,\dots,l_b)> \bl$, since $\bl$ is Lyndon. In either case we get a contradiction.  
\end{proof}

\subsection{Maximal elements in shuffle products} %We continue working with Cartan data of finite type. 
We need to emphasize a terminological issue: when we speak of a shuffle of  several  words, we mean the corresponding shuffle {\em permutation} (i.e. an individual term  appearing in the shuffle product (\ref{EQSh})) as opposed to the {\em word}  obtained as a result of the shuffle. For example, there are {\em two}\, shuffles of $(i)$ and $(i)$, but they produce {\em one}\, and the same word $(i,i)$.

\begin{Lemma}\label{LTechnical} %{\rm \cite{}}%{\bf ()}
Let $\bi\in\words$, and write the canonical factorization of $\bi$ in two ways: 
$\bi=\bi^{(1)}\dots\bi^{(k)}=(\bj^{(1)})^{n_1}\dots(\bj^{(m)})^{n_m}$, where $\bi^{(1)}\geq \dots\geq \bi^{(k)}$ and $\bj^{(1)}>\dots>\bj^{(m)}$ are Lyndon words. Then
\begin{enumerate}
\item[{\rm (i)}] $\max(\bi^{(1)}\circ \dots\circ\bi^{(k)})=\bi$.
\item[{\rm (ii)}] There are exactly $n_1!\cdots n_m!$ shuffles of $\bi^{(1)},\dots,\bi^{(k)}$ equal to $\bi$, namely the permutations of the $n_l$ words $\bj^{(l)}$ for all $l=1,\dots,m$. 
\end{enumerate}
\end{Lemma}
\begin{proof}
We apply induction on the length of the word $\bi$, the induction base being clear. Let $\bi'=\bi^{(1)}\dots\bi^{(k-1)}$. By Lemma~\ref{LLex}, to make the inductive step, it suffices to prove that $\max(\bi'\circ\bi^{(k)})=\bi$ and that the only shuffles of $\bi'$ and $\bi^{(k)}$ equal to $\bi$ are the obvious $n_m$ {\em standard insertions}\, of the word $\bi^{(k)}=\bj^{(m)}$ between the words $\bj^{(m)}$ appearing $n_{m}-1$ times in the tail of $\bi'$. 

Let $\bi^{(k)}=(i_1,\dots,i_r)$, and $\si$ be a shuffle of $\bi'$ and $\bi^{(k)}$ producing the maximal possible word $\max(\bi'\circ \bi^{(k)})$. We may assume that $\si$ shuffles the letter $i_1$ into some word $\bi^{(a)}$ for $a<k$, as otherwise $\si$ is the concatenation $\bi'\bi^{(k)}$, which is one of the standard insertions. Let $\bi^{(a)}=(j_1,\dots,j_s)$, and $u$ be the maximal index such that $i_1,\dots,i_u$ are shuffled into $\bi^{(a)}$ by $\si$, i.e. $i_1,\dots,i_u$ are shuffled to the left of $j_s$, and $i_{u+1},\dots,i_r$ are shuffled to the right of $j_s$ by $\si$. 
%By Lemma~\ref{LGoodFactor}, $(i_1,\dots,i_u)$ is a good word, so it has 
Consider the canonical factorization  $(i_1,\dots,i_u)=\bl^{(1)}\dots\bl^{(q)}$  for Lyndon words $\bl^{(1)},\dots,\bl^{(q)}$ with $(i_1,\dots,i_u)\geq \bl^{(1)}\geq \dots\geq \bl^{(q)}$. 

Assume that $k>2$ or $u<r$. 
Then the inductive assumption can be applied to the shuffle product $\bi^{(a)}\circ \bl^{(1)}\circ \dots\circ \bl^{(q)}$. If $u<r$, then $(i_1,\dots,i_u)<\bi^{(k)}$, in which case $\si$ does not produce $\max(\bi'\circ \bi^{(k)})$, since, by the inductive assumption, inserting $i_1,\dots,i_u$ immediately after the word $\bi^{(a)}$ gives to a larger word. On the other hand, if  $u=r$, 
the inductive assumption applied to $\bi^{(a)}\circ \bi^{(k)}$ implies that $\si$ can produce $\max(\bi'\circ \bi^{(k)})$ only if $\bi^{(k)}=\bi^{(a)}$ and $\bi^{(k)}$ is inserted immediately to the left of $\bi^{(a)}$, which is one of the standard insertions.

We are reduced to the case where  $k=2$ and $u=r$, so that we are shuffling $\bi^{(1)}=(j_1,\dots,j_s$ and $\bi^{(2)}=(i_1,\dots, i_r)$ and we are only considering the set of shuffles $X$ such that $i_r$ appears to the left of $j_s$. 
%By Lemma~\ref{LGoodFactor}, 
Consider the canonical factorization $(j_1,\dots,j_{s-1})=\bk^{(1)}\dots\bk^{(t)}$ for  Lyndon words $\bk^{(1)}\geq\dots\geq \bk^{(t)}$. Let $1\leq y\leq t$ be chosen so that $\bk^{(1)}\geq\dots\geq\bk^{(y-1)}\geq \bi^{(2)}\geq\bk^{(y)}\geq\dots \geq \bk^{(t)}$. By  Lemma~\ref{LLex} and the inductive assumption, we have 
\begin{align*}
\max((j_1,\dots,j_{s-1})\circ \bi^{(2)})& = \max(\bk^{(1)}\circ\dots\circ \bk^{(t)}\circ \bi^{(2)})
\\
&=\max(\bk^{(1)}\circ\dots\circ\bk^{(y-1)}\circ \bi^{(2)}\circ\bk^{(y)}\circ\dots \circ \bk^{(t)})
\\&=\bk^{(1)}\dots\bk^{(y-1)} \bi^{(2)}\bk^{(y)}\dots \bk^{(t)}.
\end{align*}
It follows that $\max(X)
= \bk^{(1)}\dots\bk^{(y-1)} \bi^{(2)}\bk^{(y)}\dots \bk^{(t)}(j_s)$. Now Lemma~\ref{LLL} implies that $\bi^{(1)}=\bi^{(2)}$ and $\si$ is the concatenation $\bi^{(2)}\bi^{(1)}$, which is one of the standard insertions. 
\end{proof}

Now it is easy to compute top coefficients in quantum shuffle products corresponding to canonical factorizations.

\begin{Corollary}\label{CTough}%{\rm \cite{}}%{\bf ()}
Let %$\bi\in\words $ be a good word with the canonical factorization  
$
\bi=(\bj^{(1)})^{n_1}\dots(\bj^{(m)})^{n_m} 
$ 
be a canonical factorization,  
with %good 
Lyndon words $\bj^{(1)}>\dots>\bj^{(m)}$. Set $\beta_k:=|\bj^{(k)}|$ for $k=1,\dots,m$. Then, for some coefficients $a_\bj\in\Laurent$, we have 
\begin{equation*}\label{EEE}
(\bj^{(1)})^{\circ n_1}\circ \dots\circ (\bj^{(m)})^{\circ n_m}=\big(\prod_{k=1}^m q_{\beta_k}^{-n_k(n_k-1)/2}[n_k]_{\beta_k}^!\big)\bi+\sum _{\bj<\bi}a_{\bj}\bj
\end{equation*}
\end{Corollary}
\begin{proof}
One just needs to observe that the transposition of the two words equal to $\bj^{(k)}$ in the shuffle product appears  with the coefficient $q_{\beta_k}^{-2}$. 
\end{proof}

\subsection{Good words and highest weights} 
Recall that we have a natural embedding ${\mathbf f}^*\subset \mathbf{'f}^*$. 
A word $\bi\in\words$ is called {\em good} if there is %a homogeneous 
$x\in {\mathbf f}^*$ such that $\bi=\max(x)$. 
%In view of \cite[Proposition 6(i)]{Lec}, this definition is equivalent to \cite[Definition 11]{Lec}. 
Denote by $\words_+$ the set of all good words, and let 
$$\words_+^\al=\words_+\cap\words^\al.$$

Given a module $L\in\Rep{R_\al}$, we say that $\bi\in\words$ is the {\em highest weight} of $L$ if $\bi=\max(\CH L)$. Theorem~\ref{Dklthm}(3) immediately gives:

\begin{Lemma}\label{LHW}%{\rm \cite{}}%{\bf ()}
The highest weight of any $L\in\Rep{R_\al}$ is a good word in $\words^\al_+$. 
\end{Lemma}

Later on, under the assumption that the Cartan datum is of finite type, we will prove that for each $\bi\in\words_+^\al$ there exists a unique up to isomorphism {\em irreducible}\,  module $L(\bi)\in\Rep{R_\al}$ with highest weight $\bi$. %We will also obtain an explicit construction of $L(\bi)$ as the head of a certain induced module. 

%Still for general type we have: 
%\begin{Lemma}\label{LGoodFactor} {\rm \cite{LR}, \cite[Lemma 13]{Lec}} %{\bf ()} Every factor of a good word is good.
%\end{Lemma}

The following result 
%, first proved in \cite{LR}, 
reduces classification of good words to that of good Lyndon words. 
%Since our definition coincides with that given in \cite{Lec} it is more convenient to use the proof of \cite[Proposition 17]{Lec} (which does not depend on Cartan datum being of finite type).

%The following Lemma reduces classification of good words to that of good Lyndon words. 

\begin{Lemma}\label{LLR1} {\rm \cite{LR}, \cite[Proposition 17]{Lec}} 
Let $\bi$ be a word with canonical factorization $\bi=\bi^{(1)}\bi^{(2)} \dots\  \bi^{(k)}$. Then $\bi$  is good if and only if the Lyndon words $\bi^{(1)},\bi^{(2)}, \dots, \bi^{(k)}$ are good.
\end{Lemma}

%\subsection{Good Lyndon words}
For Cartan data of finite type,  good Lyndon words can be classified using the following fact first noticed in {\rm \cite{LR}}, see also \cite[Proposition 18]{Lec}:

\begin{Lemma}\label{TWords}  %
If the Cartan datum is of finite type then the map $\bi\mapsto|\bi|$ is a bijection between the set of good Lyndon words and the set $\De_+$ of positive roots. 
\end{Lemma}

Denote the inverse of the bijection from Lemma~\ref{TWords} as follows: 
\begin{equation}\label{EBij}
\be\mapsto \bi(\be)\qquad(\be\in\De_+)
\end{equation}
We use the bijection (\ref{EBij}) to transport the lexicographic order on good Lyndon words to a total order `$<$' on the set $\De_+$. 
For a natural choice of the lexicographic order on $I$ the map (\ref{EBij}) is described explicitly in \cite[\S 3]{LR} and \cite[ \S 8]{Lec}.
%Since the set of good Lyndon words is lexicographically ordered, w
%This is used in the following theorem. 

\begin{Lemma}\label{LCon}%{\rm \cite{}}%{\bf ()}
Let the Cartan datum be of finite type, and let $\beta,\ga_1,\dots,\ga_r\in\De_+$ with $\beta>\ga_1\geq\dots\geq\ga_r$ and $n\in\Z_{\geq0}$. Then $n\beta\neq\ga_1+\dots+\ga_r$. 
\end{Lemma}
\begin{proof}
Let $\De_+=\{\beta_1,\dots,\be_N\}$ with 
$%\begin{equation}\label{EOrder}
\be_1>\be_2>\dots>\be_N.
$ %\end{equation}
It is pointed out in \cite[\S 4.5]{Lec} that 
%the order `$<$' on $\De_+$ arises from 
there is 
a (unique) reduced decomposition $w_0=r_{i_1}\dots r_{i_N}$ of the lengest element of the Weyl group such that 
%in the following way: 
$\be_1=r_{i_1}\dots r_{i_{N-1}}(\al_{i_N}),\be_2=r_{i_1}\dots r_{i_{N-2}}(\al_{i_{N-1}}),\dots,\be_N=\al_{i_1}$. 

Now, let $\beta=r_{i_1}\dots r_{i_s}(\al_{i_{s+1}})$. Then the element $w:=r_{i_s}\dots r_{i_1}$ of the Weyl group maps all positive roots smaller than $\beta$ to negative roots. In particular, $w(\ga_1),\dots,w(\ga_r)$ are negative. On the other hand $w(\beta)=\al_{i_{s+1}}$ is positive, which immediately implies the required result.  
%Geometrically, this corresponds to a path from the fundamental chamber $C$ to the chamber $w_0(C)$, which crosses the hyperplanes $\frak{h}_{\beta^{(n)}},\frak{h}_{\beta^{(n-1)}},\dots,\frak{h}_{\beta^{(1)}}$ orthogonal to the corresponding roots, with $\frak{h}_{\beta^{(n)}}$ crossed first, $\frak{h}_{\beta^{(n-1)}}$ crossed second, and so on, $\frak{h}_{\beta^{(1)}}$ crossed last. Since the hyperplane $\frak{h}_{\beta}$ is crossed after the hyperplanes   $\frak{h}_{\ga_1},\dots,\frak{h}_{\ga_r}$, the root $\beta$ is not in the cone $\Bbb{R}_{>0}\ga_1+\cdots+\Bbb{R}_{>0}\ga_r$.  Hence $n\beta\not\in \Bbb{R}_{>0}\ga_1+\cdots+\Bbb{R}_{>0}\ga_r$.
\end{proof}

The following result slightly generalizes \cite[Corollary 27]{Lec}. %The proof follows the same ideas. 

\begin{Lemma}\label{LFishy2}%{\rm \cite{}}%
%Assume that the Cartan datum is of finite type. 
Let the Cartan datum be of finite type, $n\in \Z_{\geq 0}$, and $\be\in\De_+$. Then $\bi(\be)^n$ is the smallest good word in $\words^{n\be}$. 
\end{Lemma}
\begin{proof}  Let $\bi = \bi(\beta)$.
Let $\bj$ be a good word of weight $n\beta$ and assume that $\bj< \bi^n$.  Let
$\bj = \bj^{(1)}\dots \bj^{(m)}$ be the canonical factorization of $\bj$.  By the discussion before Lemma 4.1 in \cite{Me},  if $\bk = \bk^{(1)}\dots \bk^{(r)}$
and $\bl = \bl^{(1)}\dots \bl^{(s)}$ are canonical factorisations 
and $\bk<\bl$, then there exists $t$ such that $\bk^{(u)} = \bl^{(u)}$ for $u<t$ and $\bk^{(t)}< \bl^{(t)}$.  In our
case this means that $\bj = \bi\dots \bi\bj^{(t)}\cdots \bj^{(m)}$ with
$\bi>\bj^{(t)}\ge \dots\geq \bj^{(m)}$. Set $\ga_r:=|\bj^{(r)}|$. Then $\beta > \ga_t\ge \dots \ge\ga_m$ and 
$\ga_t+\cdots+\ga_m = (n-t+1)\beta.$ By Lemma~\ref{LCon}, this is a contradiction. 
\end{proof}

\section{Cuspidal modules}
In this section we always assume that the Cartan datum is of finite type. 

\subsection{Dual-canonical basis}
The {\em dual-canonical basis} $\{b_{\bi}^*\mid\bi\in\words_+\}$ 
of $\mathbf{f}_\Laurent^*$ is defined in \cite{Lec}. Its elements are labeled by the set of good words $\words_+$ and can be computed using the algorithm of \cite[\S5.5]{Lec}. The first step of the algorithm computes $b_\bi^*$ when $\bi$ is a good Lyndon word, see \cite[\S5.5.2, \S5.5.4, \S8]{Lec}. 
The dual-canonical basis is well-behaved with respect to the weight space decomposition, i.e. 
\begin{equation}\label{ECanBasisWeight}
\{b_{\bi}^*\mid\bi\in\words_+^\al\}
\end{equation}
is an $\Laurent$-basis of the weight space $(\mathbf{f}_\Laurent^*)_\al$ for any $\al\in Q_+$. 

Recall that we always identify $\mathbf{f}_\Laurent^*$ with a subalgebra of $\mathbf{'f}_\Laurent^*$ using the map $\iota$, and so we consider $b_{\bi}^*$ as elements of $\mathbf{'f}_\Laurent^*$.

\begin{Lemma}\label{LMaxCan} {\rm \cite[Theorem 40]{Lec}} %{\bf ()}
For all $\bi\in\words_+$, we have $\max(b_\bi^*)=\bi$. 
\end{Lemma}

Denote the coefficient of $\bi$ in $b_\bi^*$ by $\kappa_\bi$. This coefficient is known in many situations, see e.g. \cite[\S8]{Lec}. For example, for natural orderings of simple roots in finite $A,D,E$ types we always have $\kappa_\bi=1$, see  
\cite[Proposition 56]{Lec}.

\begin{Lemma}\label{LKappa} {\rm \cite[(19)]{Lec}} %{\bf ()}
Let the canonical factorization of $\bi\in\words_+$ be written in the form $\bi=(\bj^{(1)})^{n_1}\dots(\bj^{(m)})^{n_m}$ where $\bj^{(1)}>\dots>\bj^{(m)}$ are good Lyndon words, and set $\be_k:=|\bj^{(k)}|$ for $k=1,\dots,m$. 
Then %the coefficient of $\bi$ in $b_\bi^*$ is 
$\kappa_\bi=\prod_{k=1}^m \kappa_{\bj^{(k)}}^{n_k}[n_k]_{\beta_k}^!$.
\end{Lemma}

\subsection{Cuspidal modules}
Let $\be\in \De_+$. An irreducible module $L\in\Rep{R_\be}$ is called {\em cuspidal} if its highest weight is a good Lyndon word. Recall that we have a bijection (\ref{EBij}) between $\De_+$ and good Lyndon words.  So $L$ is cuspidal if and only if its highest weight is of the form $\bi(\be)$. % for some positive root $\beta$. It follows that $\al=\be$, and so in particular cuspidal modules in $\Rep{R_\al}$ can exist only if $\al\in\De_+$. 

\begin{Lemma}%\label{}%{\rm \cite{}}%{\bf ()}
Let $\beta\in\De_+$. There is at most one cuspidal irreducible module in $\Rep{R_\be}$ (up to isomorphism and degree shift).
\end{Lemma}
\begin{proof}
Let $L_1,L_2$ be irreducible modules with highest weight $\bi(\beta)$ such that $L_1\not\cong L_2\langle m\rangle$ for all $m\in \Z$. By Corollary~\ref{CInjective}, the $q$-character map is injective, so $\CH L_1$ and $\CH L_2$ are linearly independent elements of $\mathbf{'f}^*$. So any linear combination of them is a non-zero element of $\mathbf{f}^*$. But there exists such a linear combination $x$ of $\CH L_1$ and $\CH L_2$ in which $\bi(\beta)$ does not appear, and so  $x\in \mathbf{'f}^*$ is a non-trivial linear combination of words $\bj\in \words^\beta$ such that $\bj<\bi(\beta)$. This contradicts Lemma~\ref{LFishy2} (with $n=1$) and the definition of good words. 
\end{proof}

Using the theory of dual-canonical bases we can prove a %even 
stronger result: 

\begin{Lemma}\label{LCuspCan}%{\rm \cite{}}%{\bf ()}
Let $\beta\in\De_+$ and $L\in\Rep{R_\beta}$ be a cuspidal irreducible module such that $L^\circledast\cong L$. Then $\CH L=b_{\bi(\beta)}^*$. In particular, %$\max (\CH L)=\bi(\beta)$ and 
$\qdim L_{\bi(\beta)}=\kappa_\bi$.
\end{Lemma}
\begin{proof}
By Theorem~\ref{Dklthm}(4), we have $\CH L\in (\mathbf{f}_\Laurent^*)_\beta$. By (\ref{ECanBasisWeight}),  $\{b_{\bi}^*\mid\bi\in\words_+^\beta\}$ is  
an $\Laurent$-basis of $(\mathbf{f}_\Laurent^*)_\be$, so $\CH L=\sum_{\bi\in\words_+^\beta}a_\bi b_{\bi}^*$ for some $a_\bi\in\Laurent$. By definition of the cuspidal modules and Lemmas~\ref{LFishy2} and \ref{LMaxCan} we conclude that $a_\bi=0$ unless $\bi=\bi(\beta)$, i.e. $\CH L=a_{\bi(\beta)}b_{\bi(\beta)}^*$. On the other hand, the $q$-characters of the graded irreducible $R_\be$-modules also form an $\Laurent$-basis of  $(\mathbf{f}_\Laurent^*)_\beta$, see Theorem~\ref{Dklthm}(3). So $a_{\bi(\beta)}\in\Laurent$ must be invertible, i.e. $a_{\bi(\beta)}=\pm q^n$ for some $n\in\Z$. By \cite[Proposition 39(ii)]{Lec}, the coefficients of $b_{\bi(\beta)}^*$ in the basis of words are bar-invariant, and the same is true for $\CH L$ by Theorem~\ref{Dklthm}(7). So $b_{\bi(\beta)}^*=\pm\CH L$. By definition, $\CH L$ is a $\Z_{\geq 0}[q,q^{-1}]$-combination of $\bj$'s. On the other hand, it follows from \cite[14.4.2,14.4.3]{Lubook} that at least one of the coefficients in the decomposition of $ b_{\bi(\beta)}^*$ with respect to $\bj$'s is in $\Z_{>0}[q,q^{-1}]$. This rules out $b_{\bi(\beta)}^*=-\CH L$. 
\end{proof} 

\subsection{Existence of cuspidal modules}
From now on we %are going to 
assume the following

\begin{Hypothesis}\label{Hyp}
For every $\beta\in\De_+$ there exists a cuspidal irreducible module in $\Rep{R_\beta}$. 
\end{Hypothesis}

Let us point out right away that Hypothesis~\ref{Hyp} is actually {\em true}. This follows in the simply laced case from Varagnolo-Vasserot \cite{VV} and in general from  Rouquier \cite{R2}, who prove that the classes $[L]$ of the graded irreducible $R$-modules in characteristic zero under the isomorphism $\ga^*$ from Theorem~\ref{Dklthm} correspond to the dual canonical basis elements in $\mathbf{f}_\Laurent^*$. To extend the result from a field of characteristic zero to an arbitrary field, it remains to apply Lemma~\ref{LMaxCan} and a reduction modulo $p$ argument using the fact that the algebras $R_\be$ are defined over $\Z$, see \cite{KL1}.

More importantly, in most cases it is possible to exhibit a very explicit and elementary construction of cuspidal modules. In section~\ref{SMoreCusp} we  do this in all cases except for twelve positive roots in type $E_8$ and nine positive roots in type $F_4$. This will be done for a certain natural ordering on $I$, see \S\ref{SSNO}. 
%Cuspidal modules can also be constructed explicitly in type $A$ for an arbitrary ordering of $I$. W 
We will return to this issues in section~\ref{SMoreCusp}, after the construction of irreducible $R_\al$-modules is complete. 
%, as heads of appropriate standard modules. 

In view of Hypothesis~\ref{Hyp} and Lemma~\ref{LSelfD}, for every $\beta\in\De_+$  there exists a unique up to isomorphism self-dual cuspidal irreducible $R_\beta$-module. We denote it by $L_\beta$.

\subsection{Powers of cuspidal modules}
By associativity of induction, we can iterate the notation (\ref{ECircle}). In particular if $L\in\Rep{R_\al}$ then the product $L\circ L\circ \dots \circ L$ with $n$ factors is a module in $\Rep{R_{n\al}}$ denoted $L^{\circ n}$. 

\begin{Lemma}\label{LPower}%{\rm \cite{}}%{\bf ()}
Let $\beta\in \De_+$ and $n\in\Z_{> 0}$. 
%, and $L=L_\beta$ be a self-dual cuspidal irreducible module in $\Rep{R_\be}$. 
Then $L_\beta^{\circ n}$ is an irreducible module %in $\Rep{R_{n\beta}}$ 
with highest weight $\bi(\beta)^n$, and $\CH L_\beta^{\circ n}=q_\beta^{-n(n-1)/2}b_{\bi(\beta)^n}^*$. In particular, 
$\qdim (L_\beta^{\circ n})_{\bi(\beta)^n}=q_\beta^{-n(n-1)/2}\kappa_{\bi(\beta)}^n[n]_{\beta}^!$. 
\end{Lemma}
\begin{proof}
By Lemmas~\ref{LLex}, \ref{LHW} and \ref{LFishy2}, all composition factors of $L_\beta^{\circ n}$ have highest weight $\bi(\beta)^n$. Moreover, by Lemma~\ref{LCuspCan}, $\qdim (L_\beta)_{\bi(\be)}=\kappa_{\bi(\beta)}$. %Let $i\in I$ be such that $i\cdot i=\beta\cdot\beta$. Then,  u
So, by Corollaries~\ref{CChSh} and \ref{CTough}, we have $\qdim (L_\beta^{\circ n})_{\bi(\beta)^n}=q_\beta^{-n(n-1)/2}\kappa_{\bi(\beta)}^n[n]_{\beta}^!$. 

%Lemma~\ref{LLex},  the multiplicity of $\bi(\beta)^n$ in $\CH  L_\beta^{\circ n}$ is equal to $\kappa_\beta^n[n]_\beta^!$. 

By Theorem~\ref{Dklthm}, we have $\CH L_\beta^{\circ n}\in (\mathbf{f}_\Laurent^*)_{n\beta}$. As $\{b_{\bi}^*\mid\bi\in\words_+^{n\beta}\}$ is 
an $\Laurent$-basis of $(\mathbf{f}_\Laurent^*)_{n\beta}$, we have $\CH L_\beta=\sum_{\bi\in\words_+^{n\beta}}a_\bi b_{\bi}^*$ for some $a_\bi\in\Laurent$. By %definition of the cuspidal modules and 
Lemmas~\ref{LFishy2} and \ref{LMaxCan} we conclude that $a_\bi=0$ unless $\bi=\bi(\beta)^n$. To see that $a_{\bi(\beta)^n}=q_\beta^{-n(n-1)/2}$ we compare the $\bi(\beta)^n$-coefficients using Lemma~\ref{LKappa}. 
%, i.e. $\CH L=a_{\bi(\beta)^n}b_{\bi(\beta)^n}^*$. On the other hand, the graded characters of the graded irreducible $R_{n\be}$-modules also form an $\Laurent$-basis of  $(\mathbf{f}_\Laurent^*)_{n\beta}$, see Corollary~\ref{Dklthm}(3). So $a_{\bi(\beta)^n}\in\Laurent$ must be invertible, i.e. $\pm q^n$ for some $n\in\Z$. By definition, $b_{n\bi(\beta)}^*$ is bar-invariant, and the same is true for $\CH L^{\circ n}$ by Theorem~\ref{Dklthm}(7). So we have $b_{\bi(\beta)^n}^*=\pm\CH L^{\circ n}$. Note by definition that $\CH L^{\circ n}$ is a $\Z_{\geq 0}[q,q^{-1}]$-combination of $\bj$'s. On the other hand, it follows from \cite[14.4.2,14.4.3]{Lubook} that at least one of the coefficients in the decomposition of $ b_{\bi(\beta)^n}^*$ with respect to $\bj$'s is in $\Z_{>0}[q,q^{-1}]$. This rules out $b_{\bi(\beta)^n}^*=-\CH L^{\circ n}$. 
\end{proof}

\section{Classification of irreducible modules}
We continue working with Cartan data of finite type. %In this section we construct the irreducible modules in $\Rep{R_\al}$ as simple heads of standard modules. 

\subsection{Standard modules} Let $\bi\in \words^\al_+$. 
Write the canonical factorization of $\bi$ in the form 
$
\bi=(\bj^{(1)})^{n_1}\dots(\bj^{(m)})^{n_m},
$ 
where $\bj^{(1)}>\dots>\bj^{(m)}$ are good Lyndon words. Set $\beta_k:=|\bj^{(k)}|\in\De_+$ for $k=1,\dots,m$. 
Denote 
$$
s(\bi):=\sum_{k=1}^m(\beta_k\cdot\beta_k)n_k(n_k-1)/4,
$$
cf. \cite[\S5.5.3]{Lec}. Define the {\em standard module} of highest weight $\bi$ as the following induced $R_{\al}$-module with a degree shift: 
$$
\De(\bi):=L_{\be_1}^{\circ n_1}\circ\dots\circ L_{\be_m}^{\circ n_m}\langle s(\bi)\rangle.
$$
%Note that $\De(\bi)$ is an $R_{|\bi|}$-module. 

\begin{Lemma}\label{LDelta}%{\rm \cite{}}%{\bf ()}
Let $\bi\in\words_+$. Then $\max(\CH \De(\bi))=\bi$, and $\qdim \De(\bi)_\bi=\kappa_\bi$. 
\end{Lemma}
\begin{proof}
Adopt the notation of the paragraph preceding the lemma. By Corollary~\ref{CChSh}, we have $\CH\De(\bi)=q^{s(\bi)}(\CH L_{\beta_1})^{\circ n_1}\circ \dots\circ (\CH L_{\beta_m})^{\circ n_m}$. So by Lemma~\ref{LLex} the maximal weight can only come from the shuffle product of the highest weights
$
(\bj^{(1)})^{\circ n_1}\circ\dots\circ (\bj^{(m)})^{\circ n_m}. 
$ 
By Lemma~\ref{LCuspCan},  the multiplicity of $\bj^{(k)}$ in $L_{\beta_k}$ is $\kappa_{\bj^{(k)}}$.  So, by Corollary~\ref{CTough} we have 
%$\bi=\max\big(\CH\De(\bi))$ and 
$$
\CH \De(\bi)=\Big(q^{s(\bi)}\prod_{k=1}^m \kappa_{\bj^{(k)}}q_{\beta_k}^{-n_k(n_k-1)/2}[n_k]_{\beta_k}^!\Big)\bi+\sum_{\bj<\bi}a_{\bj}\bj.
$$
Since $q^{s(\bi)}=\prod_{k=1}^mq_{\beta_k}^{n_k(n_k-1)/2}$, we get 
$\qdim \De(\bi)_\bi=\prod_{k=1}^m \kappa_{\bj^{(k)}}[n_k]_{\beta_k}^!$, which is $\kappa_\bi$ in view of Lemma~\ref{LKappa}. 
\end{proof}

\subsection{Irreducible modules} Now we prove our main result.

\begin{Theorem}\label{TMainIrr}%{\rm \cite{}}%{\bf ()}
Let $\al\in Q_+$ and $\bi\in\words_+^\al$. Then: 
\begin{enumerate}
\item[{\rm (i)}]  The standard $R_\al$-module $\De(\bi)$ has an irreducible head $L(\bi)$.
\item[{\rm (ii)}] The highest weight of $L(\bi)$ is $\bi$, and $\qdim L(\bi)_\bi=\kappa_\bi$. 
\item[{\rm (iii)}] $L(\bi)^\circledast\cong L(\bi)$. 
\item[(iv)] $\{L(\bj)\mid \bj\in\words_+^\al\}$ is a complete and irredundant set of irreducible graded $R_\al$-modules up to isomorphism and degree shift.
\item[(v)] If $\bi=\bj^n$ for a good Lyndon word $\bj$, then $L(\bi)=\De(\bi)$. 
\end{enumerate} 
\end{Theorem}
\begin{proof}
Part (v) is contained in Lemma~\ref{LPower}. Write $
\bi=(\bj^{(1)})^{n_1}\dots(\bj^{(m)})^{n_m},
$ where $\bj^{(1)}>\dots>\bj^{(m)}$ are good Lyndon words. Set $\beta_k:=|\bj^{(k)}|\in\De_+$ for $k=1,\dots,m$. 
Let $L\in\Rep{R_\al}$ be irreducible. If $L$ is in the head of $\De(\bi)$ then we have using Frobenius reciprocity: 
\begin{align*}
0&\neq \Hom_{R_\al}(\De(\bi),L)
\\
&=\Hom_{R_{n_1\beta_1,\dots,n_m\beta_m}}(L_{\beta_1}^{\circ n_1}\boxtimes\dots\boxtimes L_{\beta_m}^{\circ n_m}\langle s(\bi)\rangle,\Res^\al_{n_1\beta_1,\dots,n_m\beta_m} L). 
\end{align*}
By Lemma~\ref{LPower}, the $R_{n_1\beta_1,\dots,n_m\beta_m}$-module $L_{\beta_1}^{\circ n_1}\boxtimes\dots\boxtimes L_{\beta_m}^{\circ n_m}\langle s(\bi)\rangle$ is irreducible, and so it embeds into $\Res^\al_{n_1\beta_1,\dots,n_m\beta_m}L$. By Lemma~\ref{LPower} applied to each $L_{\beta_k}^{\circ n_k}$ and Lemma~\ref{LDelta}, the multiplicity of the weight $\bi$ in the submodule $L_{\beta_1}^{\circ n_1}\boxtimes\dots\boxtimes L_{\beta_m}^{\circ n_m}\langle s(\bi)\rangle$ of $\Res^\al_{n_1\beta_1,\dots,n_m\beta_m}\De(\bi)$ is equal to the multiplicity of the weight $\bi$ in $\De(\bi)_\bi$. As  $L_{\beta_1}^{\circ n_1}\boxtimes\dots\boxtimes L_{\beta_m}^{\circ n_m}\langle s(\bi)\rangle$ also embeds into $L$ we conclude that $\qdim L_\bi=\qdim \De(\bi)_\bi=\kappa_\bi$. Hence the head of $\De(\bi)$ is irreducible. Moreover, 
as $\bi$ is the highest weight of $\De(\bi)$ it is also the highest weight of $L$. 

We have proved (i) and (ii). To see (iii), recall from Lemma~\ref{LSelfD} that there always exists a self-dual shift $L(\bi)\langle n\rangle$. As $\kappa_\bi=\qdim L(\bi)_\bi$ is bar-invariant, we conclude that $n=0$. Finally, part (iv) is obtained by a counting argument. Indeed, we have constructed $|\words_+^\al|$ non-isomorphic irreducible graded $R_\al$-modules $\{L(\bj)\mid \bj\in\words_+^\al\}$. On the other hand, 
the basis (\ref{ECanBasisWeight}) of $(\mathbf f^*_\Laurent)_\al$ is labeled by the same set $\words_+^\al$. Now, to see that we have constructed all irreducible graded $R_\al$-modules, it suffices to use Theorem~\ref{Dklthm}(4). 
\end{proof}

It is crucial in the following conjecture that the Cartan datum is assumed to be of finite type. Analogous statement is known to be false in the affine type $A$, where a more subtle James Conjecture suggests an answer to a similar question. We refer the reader e.g. to \cite{Kbull} for details on that. 

\begin{Conjecture}\label{Conj}
Let the Cartan datum be of finite type and $\al\in Q_+$. 
The formal characters of the irreducible graded $R_\al$-modules are  independent of the characteristic of the ground field $\FF$.
\end{Conjecture}

As pointed out by Brundan, the conjecture is true in finite type $A$ for representations which factor through a cyclotomic Khovanov-Lauda-Rouquier algebra of level two. This follows from \cite[Lemma 9.7 and Remark 8.7]{BS}. 

\subsection{Graded decomposition numbers} 
Let $\al\in Q_+$ and $\bi,\bj\in \words_+^\al$. The corresponding {\em graded decomposition number} is defined to be the graded composition multiplicity (see (\ref{EGCM})): 
$$
d_{\bi,\bj}:=[\De(\bi):L(\bj)]_q.
$$

\begin{Proposition}%\label{}%{\rm \cite{}}%{\bf ()}
Let $\al\in Q_+$. Then: 
\begin{enumerate}
\item[{\rm (i)}] $d_{\bi,\bj}\in\Z_{\geq 0}[q,q^{-1}]$ for all $\bi,\bj\in \words_+^\al$; 
\item[{\rm (ii)}] the graded decomposition matrix $D^\al:=(d_{\bi,\bj})_{\bi,\bj\in \words_+^\al}$ is uni-triangualr, i.e. $d_{\bi,\bi}=1$ and $d_{\bi,\bj}=0$ unless $\bj\leq \bi$. 
\end{enumerate}
\end{Proposition}
\begin{proof}
(i) follows from the definition of the graded composition multiplicity and (ii) follows from Theorem~\ref{TMainIrr}. 
\end{proof}

The following corollary generalizes Theorem~\ref{TMainIrr}(v). 

\begin{Corollary}\label{CArun}%{\rm \cite{}}%{\bf ()}
Let $\al\in Q_+$, and let $\bi$ be the minimal element of $\words_+^\al$. Then $\De(\bi)=L(\bi)$. 
\end{Corollary}

Motivated by Corollary~\ref{CArun} and \cite{LNT} it is reasonable to consider the following 

\begin{Problem} 
\begin{enumerate}
\item[{\rm (i)}] Describe all $\bi\in \words_+$ such that $\De(\bi)=L(\bi)$.
\item[{\rm (ii)}] Describe all $\bi,\bj\in \words_+$ such that $L(\bi)\circ L(\bj)$ is irreducible.
\end{enumerate}
\end{Problem}

\section{Cuspidal modules}\label{SMoreCusp}
For natural orderings of the set $I$ it should always be possible to exhibit explicit constructions of the cuspidal modules. In this section we explain how to do this in all cases except nine positive roots in type $F_4$ and twelve positive roots in type $E_8$. The missing cases should be easy to handle using computer. In simply laced cases, except for the twelve exceptional roots in type $E_8$, the cuspidal modules always belong to an especially nice class of modules studied in \cite{KRhomog} and called homogeneous modules. 

\subsection{Homogeneous modules}\label{SShomog}  
Throughout this subsection we assume that the Cartan datum is of simply laced type. 
A graded $R_\al$-module is called {\em homogeneous} if it is concentrated in one degree. Homogeneous modules have been studied in \cite{KRhomog}, from where 
we cite some necessary results. 

As usual, we work with an arbitrary fixed $\al\in Q_+$ of height $d$. Let  $\bi\in \words^\al$. We call $s_r\in S_d$ an {\em admissible transposition} for $\bi$ if $a_{i_r, i_{r+1}}=0$. 
The {\em weight graph} $G_\al$ is the graph with the set of vertices $\words^\al$, and with $\bi,\bj\in \words^\al$ connected by an edge if and only if $\bj=s_r \bi$ for some admissible transposition $s_r$ for $\bi$. 
Explicit combinatorial descriptions of the connected components of $G_\al$  can be found in \cite{KRhomog}. 

Let $C$ be a connected component of $G_\al$. We say that $C$ is {\em homogeneous} if for each $\bi\in C$ the following condition holds:
\begin{equation}\label{ENC}
\begin{split}
\text{if $i_r=i_s$ for some $r<s$ then there exist $t,u$}
\\  
\text{such that $r<t<u<s$ and $a_{i_r,i_t}=a_{i_r,i_u}=-1$.}
\end{split}
\end{equation}
We say that $C$ is {\em strongly homogeneous} if for each $\bi\in C$ the following two conditions hold:
\begin{equation}\label{ENC1}
\begin{split}
\text{if $i_r=i_s$ for some $r<s$ and $i_m\neq i_r$ for all $r<m<s$}
\\ \text{then there exist exactly two indices $t,u$ such that}
\\  
\text{$r<t<u<s$ and $a_{i_r,i_t}=a_{i_r,i_u}=-1$;}
\end{split}
\end{equation}
\begin{equation}\label{ENC2}
\begin{split}
\text{if $i_r\neq i_s$ for all $r<s$ then there exists at most one $t<s$}
\\  
\text{such that $a_{i_t,i_s}=-1$.}
\end{split}
\end{equation}

\begin{Lemma}\label{LHomComp} %{\bf ()}
Let $C$ be a connected component of $G_\al$.
\begin{enumerate}
\item[{\rm (i)}] \cite[Proposition 2.1]{S2}, \cite[\S3.5]{KRhomog} If $C$ is strongly homogeneous then it is homogeneous;
\item[{\rm (ii)}]  {\rm \cite[Lemma 3.3]{KRhomog}} $C$ is homogeneous if and only if the condition (\ref{ENC}) holds for {\em some} $\bi\in C$; 
\item[{\rm (iii)}]  \cite[Proposition 2.5]{S2}, \cite[\S3.5]{KRhomog} $C$ is strongly homogeneous if and only if the conditions (\ref{ENC1}) and (\ref{ENC2}) hold for {\em some} $\bi\in C$.
\end{enumerate}
\end{Lemma}

Let $W$ be the Weyl group corresponding to our fixed Cartan datum with simple reflections $\{r_i\mid i\in I\}$. 
Let $C$ be a homogeneous connected component of $G_\al$. 
Pick $\bi=(i_1,\dots,i_d)\in C$ and define the element $w_C\in W$ as follows: $w_C:=r_{i_d}r_{i_{d-1}}\dots r_{i_1}$. Then $w_C$ depends only on $C$ and not on the chosen representative $\bi\in C$, cf. \cite[Proposition 3.7]{KRhomog}. 

The main theorem on homogeneous representations is:

\begin{Theorem}\label{Thomog} {\rm \cite[Theorems 3.6, 3.10]{KRhomog}} %{\bf ()}

\begin{enumerate}
\item[{\rm (i)}] Let $C$ be a homogeneous connected component of $G_\al$. Let $S(C)$ be the vector space  concentrated in degree $0$ with basis $\{v_\bi\mid \bi\in C\}$ labeled by the elements of $C$. 
The formulas
\begin{align*}
e(\bj)v_\bi&=\de_{\bi,\bj}v_\bi \qquad (\bj\in \words^\al,\ \bi\in C),\\
y_r v_\bi&=0\qquad (1\leq r\leq d,\ \bi\in C),\\
\psi_rv_{\bi}&=
\left\{
\begin{array}{ll}
v_{s_r\bi} &\hbox{if $s_r\bi\in C$,}\\
0 &\hbox{otherwise;}
\end{array}
\right.
\quad(1\leq r<d,\ \bi\in C)
\end{align*}
define an action of $R_\al$ on $S(C)$, under which $S(C)$ is a homogeneous irreducible $R_\al$-module. 
\item[{\rm (ii)}] $S(C)\not\cong S(C')$ if $C\neq C'$, and every homogeneous irreducible $R_\al$-module, up to a degree shift, is isomorphic to one of the modules $S(C)$
\item[{\rm (iii)}] If $C$ is strongly homogeneous then the dimension of $S(C)$ is given by the Peterson-Proctor hook formula:
$$
\dim S(C) 
= \frac{d!}{\prod_{\beta\in \Phi(C)}\height(\beta)},
$$
where
$
\Phi(C):=\{\be\in \De_+\mid w_C^{-1}(\be)\in-\De_+\}.
$
\end{enumerate}
\end{Theorem}

%\subsection{Standard labelings}\label{SSLab}

\subsection{Natural orderings}\label{SSNO}
If $|I|=\ell$, we denote its elements by $\{1,2,\dots,\ell\}$ where it is understood that $1<2<\dots<\ell$. For each finite type define the {\em natural ordering} on $I$ by:
%The natural labeling of the simple roots is as follows. 
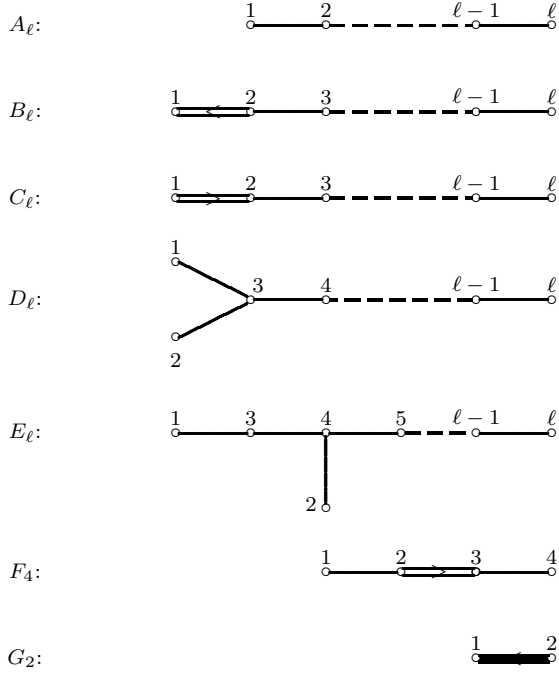
\begin{align*}
%&\hbox{Type $A$:} 
&\beginpicture
\setcoordinatesystem units <1cm,1cm>        % sets scale
%********************************************************************
% Dynkin diagram of type A_r
%********************************************************************
\setplotarea x from -3 to 3, y from -0.5 to 0.5  % sets plot size up
{\scriptsize
%\multiput {$\circ$} at 9 1.5 *1 0 1 /      %
\multiput {$\circ$} at 1   0 *1 1 0 /      %
\multiput {$\circ$} at -2   0 *1 1 0 /      %  puts nodes in
%\put {$\alpha_n$}     at 3 0.2   %
\put {$A_\ell$:}     at -5 0
\put {$\ell$}     at 2 0.2   %
\put {$\ell-1$} at 1 0.2   %
\put {$2$}     at -1 0.2   %
\put {$1$}     at -2 0.2   %
%\put {$\alpha_1$}     at -3 0.2   % label nodes with roots below
%\put{$<$} at -2.5 0
\linethickness=1pt                          % sets line thickness
%\putrule from -2.97 0.045 to -2.03 0.045       % puts solid lines between nodes
%\putrule from -2.97 -0.045 to -2.03 -0.045       %
%\putrule from -1.95 0 to -1.05 0              %
\plot -1.95 0  -1.05 0 /
%\putrule from 1.05 0 to 1.95 0              %
\plot 1.05 0 1.95 0 /
%\putrule from 2.03 0.045 to 2.97 0.045       % puts solid lines between nodes
%\putrule from 2.03 -0.045 to 2.97 -0.045       %
\setlinear
%\plot 8.95 1.5   8.05 1.95 /%
%\plot 8.95 2.5   8.05 2.05 /     %
\setdashes <2mm,1mm>          %
\putrule from -0.95 0 to 0.95 0  % draws dotted lines between nodes
}
% end of dynkin diagram of type A_r
\endpicture
%\qquad\hbox{with $\alpha_i = \varepsilon_i-\varepsilon_{i-1}$,}
%$$
%$$
\\
%&\hbox{Type $B$:}
&\beginpicture
\setcoordinatesystem units <1cm,1cm>        % sets scale
%********************************************************************
% Dynkin diagram of type B_n
%********************************************************************
\setplotarea x from -3 to 3, y from -0.5 to 0.5  % sets plot size up
{\scriptsize
%\multiput {$\circ$} at 9 1.5 *1 0 1 /      %
\multiput {$\circ$} at 1   0 *1 1 0 /      %
\multiput {$\circ$} at -3   0 *2 1 0 /      %  puts nodes in
%\put {$\alpha_n$}     at 3 0.2   %
\put {$B_\ell$:}     at -5 0
\put {$\ell$}     at 2 0.2   %
\put {$\ell-1$} at 1 0.2   %
\put {$3$}     at -1 0.2   %
\put {$2$}     at -2 0.2   %
\put {$1$}     at -3 0.2   % label nodes with roots below
\put{${\boldmath <}$} at -2.5 0
\linethickness=1pt                          % sets line thickness
%\putrule from -2.97 0.045 to -2.03 0.045       % puts solid lines between nodes
\plot -2.97 0.045  -2.03 0.045 /
%\putrule from -2.97 -0.045 to -2.03 -0.045       %
\plot -2.97 -0.045  -2.03 -0.045 /
%\putrule from -1.95 0 to -1.05 0              %
\plot -1.95 0  -1.05 0 /
%\putrule from 1.05 0 to 1.95 0              %
\plot 1.05 0  1.95 0 /
%\putrule from 2.03 0.045 to 2.97 0.045       % puts solid lines between nodes
%\putrule from 2.03 -0.045 to 2.97 -0.045       %
\setlinear
%\plot 8.95 1.5   8.05 1.95 /%
%\plot 8.95 2.5   8.05 2.05 /     %
\setdashes <2mm,1mm>          %
\putrule from -0.95 0 to 0.95 0  % draws dotted lines between nodes
}
% end of dynkin diagram of type B_n
\endpicture
%\qquad\hbox{with $\alpha_1 = \varepsilon_1$ and $\alpha_i = \varepsilon_i-\varepsilon_{i-1}$,}
%$$
\\
%&\hbox{Type $C$:}
&\beginpicture
\setcoordinatesystem units <1cm,1cm>        % sets scale
%********************************************************************
% Dynkin diagram of type B_n
%********************************************************************
\setplotarea x from -3 to 3, y from -0.5 to 0.5  % sets plot size up
{\scriptsize
%\multiput {$\circ$} at 9 1.5 *1 0 1 /      %
\multiput {$\circ$} at 1   0 *1 1 0 /      %
\multiput {$\circ$} at -3   0 *2 1 0 /      %  puts nodes in
%\put {$\alpha_n$}     at 3 0.2   %
\put {$C_\ell$:}     at -5 0
\put {$\ell$}     at 2 0.2   %
\put {$\ell-1$} at 1 0.2   %
\put {$3$}     at -1 0.2   %
\put {$2$}     at -2 0.2   %
\put {$1$}     at -3 0.2   % label nodes with roots below
\put{${\boldmath >}$} at -2.5 0
\linethickness=1pt                          % sets line thickness
%\putrule from -2.97 0.045 to -2.03 0.045       % puts solid lines between nodes
\plot -2.97 0.045 -2.03 0.045 /
%\putrule from -2.97 -0.045 to -2.03 -0.045       %
\plot -2.97 -0.045  -2.03 -0.045 /
%\putrule from -1.95 0 to -1.05 0              %
\plot -1.95 0  -1.05 0 /
%\putrule from 1.05 0 to 1.95 0              %
\plot 1.05 0  1.95 0 /
%\putrule from 2.03 0.045 to 2.97 0.045       % puts solid lines between nodes
%\putrule from 2.03 -0.045 to 2.97 -0.045       %
\setlinear
%\plot 8.95 1.5   8.05 1.95 /%
%\plot 8.95 2.5   8.05 2.05 /     %
\setdashes <2mm,1mm>          %
\putrule from -0.95 0 to 0.95 0  % draws dotted lines between nodes
}
% end of dynkin diagram of type B_n
\endpicture
%\qquad\hbox{with $\alpha_1 = \varepsilon_1$ and $\alpha_i = \varepsilon_i-\varepsilon_{i-1}$,}
\\
%&\hbox{Type $D$:}
&\beginpicture
\setcoordinatesystem units <1cm,1cm>        % sets scale
%********************************************************************
% Dynkin diagram of type D_r
%********************************************************************
\setplotarea x from -9 to -3, y from 1 to 2.7  % sets plot size up
{\scriptsize
\multiput {$\circ$} at -9 1.5 *1 0 1 /      %
\multiput {$\circ$} at -8   2 *1 1 0 /      %
\multiput {$\circ$} at -5   2 *1 1 0 /      %  puts nodes in
\put {$D_\ell$:}     at -11 2
\put {$2$}     at -9 1.2   %
\put {$1$}     at -9 2.7   %
\put {$3$} at -7.9 2.2 %
\put {$4$} at -7 2.2   %
\put {$\ell-1$}     at -5 2.2   %
\put {$\ell$}     at -4 2.2   %
%\put {$T_n$}     at -3 2.2   % label nodes with roots below
%\put{$>$} at -3.5 2
\plot -8.95 1.5   -8.05 1.95 /%
%\putrule from -8.95 1.5 to  -8.05 1.95
\plot -8.95 2.5   -8.05 2.05 / 
\linethickness=1pt                          % sets line thickness
%\putrule from -3.03 2.045 to -3.97 2.045       % puts solid lines between nodes
%\putrule from -3.03 1.955 to -3.97 1.955       %
%\putrule from -4.05 2 to -4.95 2
\plot   -4.05 2 -4.95 2   /         %
%\putrule from -7.05 2 to -7.95 2
\plot -7.05 2  -7.95 2     /     
%\putrule from -8.95 1.5 to  -8.05 1.95    %
\setlinear
%\plot -8.95 1.5   -8.05 1.95 /%
%\plot -8.95 2.5   -8.05 2.05 /     %
\setdashes <2mm,1mm>          %
\putrule from -5.05 2 to -6.95 2  % draws dotted lines between nodes
}
% end of dynkin diagram of type D_r
\endpicture
\\
&\beginpicture
\setcoordinatesystem units <1cm,1cm>        % sets scale
%********************************************************************
% Dynkin diagram of type E_r
%********************************************************************
\setplotarea x from -3 to 3, y from -1.2 to 0.5  % sets plot size up
{\scriptsize
%\multiput {$\circ$} at 9 1.5 *1 0 1 /      %
\multiput {$\circ$} at 1   0 *1 1 0 /      %
%\multiput {$\circ$} at -2   0 *1 1 0 /      %  puts nodes in
\multiput {$\circ$} at -3   0 *2 1 0 /  
\put {$\circ$} at -1  -1    
\put {$\circ$} at 0  0 
%\put {$\alpha_n$}     at 3 0.2   %
\put {$E_\ell$:}     at -5 0
\put {$\ell$}     at 2 0.2   %
\put {$\ell-1$} at 1 0.2   %
\put {$4$}     at -1 0.2   %
\put {$3$}     at -2 0.2   %
\put {$1$}     at -3 0.2   %
\put {$2$}     at -1.2 -0.95   %
\put {$5$}     at 0 0.2   %
%\put {$\alpha_1$}     at -3 0.2   % label nodes with roots below
%\put{$<$} at -2.5 0
\linethickness=1pt                          % sets line thickness
%\putrule from -2.97 0.045 to -2.03 0.045       % puts solid lines between nodes
%\putrule from -2.97 -0.045 to -2.03 -0.045       %
%\putrule from -1.95 0 to -1.05 0              %
\plot -1.95 0  -1.05 0 /
\plot -2.95 0  -2.05 0 /
\plot -0.95 0  -0.05 0 /
\plot -1 -0.05  -1 -0.93 /
%\putrule from 1.05 0 to 1.95 0              %
\plot 1.05 0 1.95 0 /
%\putrule from 2.03 0.045 to 2.97 0.045       % puts solid lines between nodes
%\putrule from 2.03 -0.045 to 2.97 -0.045       %
\setlinear
%\plot 8.95 1.5   8.05 1.95 /%
%\plot 8.95 2.5   8.05 2.05 /     %
\setdashes <2mm,1mm>          %
\putrule from 0.05 0 to 0.95 0  % draws dotted lines between nodes
}
% end of dynkin diagram of type A_r
\endpicture
%\qquad\hbox{with $\alpha_i = \varepsilon_i-\varepsilon_{i-1}$,}
%$$
%$$
\\
&\beginpicture
\setcoordinatesystem units <1cm,1cm>        % sets scale
%********************************************************************
% Dynkin diagram of type F_4
%********************************************************************
\setplotarea x from -3 to 3, y from -0.5 to 0.5  % sets plot size up
{\scriptsize
%\multiput {$\circ$} at 9 1.5 *1 0 1 /      %
\multiput {$\circ$} at 1   0 *1 1 0 /      %
\multiput {$\circ$} at -1   0 *2 1 0 /      %  puts nodes in
%\put {$\alpha_n$}     at 3 0.2   %
\put {$F_4$:}     at -5 0
\put {$4$}     at 2 0.2   %
\put {$3$} at 1 0.2   %
\put {$2$}     at 0 0.2   %
\put {$1$}     at -1 0.2   %
%\put {$1$}     at -3 0.2   % label nodes with roots below
\put{${\boldmath >}$} at 0.5 0
\linethickness=1pt                          % sets line thickness
%\putrule from -2.97 0.045 to -2.03 0.045       % puts solid lines between nodes
\plot 0.03 0.045 0.97 0.045 /
%\putrule from -2.97 -0.045 to -2.03 -0.045       %
\plot 0.03 -0.045  0.97 -0.045 /
%\putrule from -1.95 0 to -1.05 0              %
%\plot -1.95 0  -1.05 0 /
%\putrule from 1.05 0 to 1.95 0              %
\plot 1.05 0  1.95 0 /
\plot -0.95 0  -0.05 0 /
%\putrule from 2.03 0.045 to 2.97 0.045       % puts solid lines between nodes
%\putrule from 2.03 -0.045 to 2.97 -0.045       %
\setlinear
%\plot 8.95 1.5   8.05 1.95 /%
%\plot 8.95 2.5   8.05 2.05 /     %
%\setdashes <2mm,1mm>          %
%\putrule from -0.95 0 to 0.95 0  % draws dotted lines between nodes
}
% end of dynkin diagram of type B_n
\endpicture
%\qquad\hbox{with $\alpha_1 = \varepsilon_1$ and $\alpha_i = \varepsilon_i-\varepsilon_{i-1}$,}
\\
&\beginpicture
\setcoordinatesystem units <1cm,1cm>        % sets scale
%********************************************************************
% Dynkin diagram of type G_2
%********************************************************************
\setplotarea x from -3 to 3, y from -0.5 to 0.5  % sets plot size up
{\scriptsize
%\multiput {$\circ$} at 9 1.5 *1 0 1 /      %
\multiput {$\circ$} at 1   0 *1 1 0 /      %
%\multiput {$\circ$} at -1   0 *2 1 0 /      %  puts nodes in
%\put {$\alpha_n$}     at 3 0.2   %
\put {$G_2$:}     at -5 0
\put {$2$}     at 2 0.2   %
\put {$1$} at 1 0.2   %
%\put {$2$}     at 0 0.2   %
%\put {$1$}     at -1 0.2   %
%\put {$1$}     at -3 0.2   % label nodes with roots below
\put{${\boldmath <}$} at 1.5 0
\linethickness=0.82pt                          % sets line thickness
%\putrule from -2.97 0.045 to -2.03 0.045       % puts solid lines between nodes
\putrule from 1.03 0.045 to 1.97 0.045 
%\putrule from -2.97 -0.045 to -2.03 -0.045       %
\putrule from 1.03 -0.045 to 1.97 -0.045 
%\putrule from -1.95 0 to -1.05 0              %
%\plot -1.95 0  -1.05 0 /
%\putrule from 1.05 0 to 1.95 0              %
\putrule from 1.05 0 to 1.95 0 
%\plot -0.95 0  -0.05 0 /
%\putrule from 2.03 0.045 to 2.97 0.045       % puts solid lines between nodes
%\putrule from 2.03 -0.045 to 2.97 -0.045       %
\setlinear
%\plot 8.95 1.5   8.05 1.95 /%
%\plot 8.95 2.5   8.05 2.05 /     %
%\setdashes <2mm,1mm>          %
%\putrule from -0.95 0 to 0.95 0  % draws dotted lines between nodes
}
% end of dynkin diagram of type B_n
\endpicture
%\qquad\hbox{with $\alpha_1 = \varepsilon_1$ and $\alpha_i = \varepsilon_i-\varepsilon_{i-1}$,}
\end{align*}
From now on we always assume that the elements of $I$ are naturally ordered.

\subsection{Cuspidal modules as homogeneous modules}
%In this section we always assume that the elements of $I$ are naturally ordered, see \S\ref{SSNO}. 

Good Lyndon words for the natural orderings were described in \cite{LR} and \cite{Lec}. 
Define the positive roots  $\de:=2\al_1+2\al_2+3\al_3+4\al_4+3\al_5+2\al_6+\al_7+\al_8$ and $\eps:=2\al_1+2\al_2+3\al_3+4\al_4+3\al_5+2\al_6+\al_7$ in type $E_8$. Also set  
\begin{equation}\label{EExc}
\mathcal E:=\{\text{positive roots in type $E_8$ whose $\al_1$-coefficient is $2$}\}\setminus\{\de,\eps\}.
\end{equation}
The set $\mathcal E$ contains twelve roots \cite[Table VII]{Bo}.

%Recall that homogeneous components were defined in (\ref{ENC}). 
Direct inspection of the good Lyndon words, using (\ref{ENC}), shows: 

\begin{Lemma}\label{LLec} {\rm \cite[\S8.6]{Lec}} 
Let the Cartan datum be of finite $A,D,E$ type. Then the component of a good Lyndon word $\bi(\al)$ for $\al\in\De_+$ is homogeneous, except for the good Lyndon words corresponding to the twelve positive roots $\al\in \mathcal E$ in type $E_8$. 
\end{Lemma}

Note that the exclusion of the twelve exceptional roots in $\mathcal E$ corrects the statement of \cite[Proposition 56]{Lec}. With this correction the proof of  \cite[Proposition 56]{Lec} remains valid for all other roots. 

For $\al\in\De_+$ denote by $C(\al)$ the connected component of $\bi(\al)$ in $G_\al$. In view of Lemmas~\ref{LLec} and \ref{LHomComp}, $C(\al)$ is homogeneous, except for twelve roots $\al\in\mathcal E$ in type $E_8$, so we have the corresponding homogeneous irreducible $R_\al$-module $S(C(\al))$ coming from Theorem~\ref{Thomog}.

\begin{Proposition}\label{PCuspHomog}%{\rm \cite{}}%{\bf ()}
Let the Cartan datum be of finite $A,D,E$ type, and $\al\in\De_+$. Exclude the twelve positive roots $\al\in\mathcal E$ in type $E_8$. The homogeneous module $S(C(\al))$ is the cuspidal module corresponding to $\al$. Moreover, $\kappa_{\bi(\al)}=1$. 
\end{Proposition}
\begin{proof}
By \cite[Proposition 56 and Theorem 40]{Lec}, we have that $\bi(\al)$ is the highest weight of $S(C(\al))$. Now, by Theorem~\ref{Thomog}, $\kappa_{\bi(\al)}=1$. 
\end{proof}

Proposition~\ref{PCuspHomog} together with Theorem~\ref{Thomog} give an explicit construction of most cuspidal modules in finite simply laced types for the natural ordering of $I$. In the following sections we spell out some more explicit information on cuspidal modules, including non-simply-laced types.

\subsection{\boldmath Type $A$}\label{SSA}
The set of positive roots is 
$$\De_+=\{\al(m,n):=\al_m+\al_{m+1}+\dots+\al_n\mid 1\leq m\leq n\leq \ell\}.$$
Let $1\leq m\leq n\leq \ell$. 
The corresponding cuspidal module 
$$L_{\al(m,n)}=\FF\cdot v_{\al(m,n)}$$ 
is $1$-dimensional with the action of the generators of $R_{\al(m,n)}$ on the basis vector $v_{\al(m,n)}$ given as follows:
\begin{align*}
&e(m,m+1,\dots,n)v_{\al(m,n)}=v_{\al(m,n)},
\\
&e(\bj)v_{\al(m,n)}=0\qquad(\bj\neq(m,m+1,\dots,n)),
\\
&\psi_rv_{\al(m,n)}=y_sv_{\al(m,n)}=0.
\end{align*}

One can recognize Zelevinsky's segments here. Using the criterion of Lemma~\ref{LHomComp} one gets that in type $A$ all cuspidal modules are homogeneous for {\em any} of the $\ell!$ different lexicographic orders on $I$. In view of Theorem~\ref{Thomog} we thus have an explicit description of the cuspidal modules for any of these orders, and hence Theorem~\ref{TMainIrr} yields $\ell!$ different classifications of the irreducible modules over the Khovanov-Lauda-Rouquier algebra  of finite type $A$, generalizing Bernstein-Zelevinsky multisegment classification. In other types, we of course also have $\ell!$ different classifications, except that explicit constructions of the cuspidal modules are currently only available for the natural orderings. 

Let us now explain in more detail how to construct the cuspidal module $L_{\al(m,n)}$ corresponding to the root $\al(m,n)$ for $1\leq m\leq n\leq \ell$ for an {\em arbitrary ordering} of $I$. We still identify $I$ with the set $\{1,2,\dots,\ell\}$, and  $1,\dots,\ell$ still label the nodes of the Dynkin diagram from left to right, but now they could be ordered arbitrarily. 
We say that elements $i,j\in I$ are {\em neighbors} if $|i-j|=1$ or equivalently $a_{i,j}=-1$. More generally, given a subset $X\subset I$, an element $i\in I\setminus X$ is a  {\em neighbor} of $X$ if it is a neighbor of some $j\in X$. 
Define the word $\bi(m,n)=(i_1,\dots,i_{m-n+1})$ inductively as follows: set $i_1$ to be the smallest element of $\{m,m+1,\dots,n\}$; if $i_1,\dots,i_r$ have already been defined for some $1\leq r\leq m-n$ choose $i_{r+1}$ to be the biggest element of $\{m,m+1,\dots,n\}$ which is a neighbor of $\{i_1,\dots,i_r\}$. It is clear that $|\bi(m,n)|=\al(m,n)$. Let $C$ be the connected component of $G_{\al(m,n)}$ containing $\bi(m,n)$, and $L$ be the corresponding homegeneous $R_{\al(m,n)}$-module, see section~\ref{SShomog}. 
The word $\bi(m,n)$ is good since it is the largest word appearing in $\CH L$. Moreover, $\bi(m,n)$ is Lyndon since $i_1$ is smaller than $i_2,\dots,i_{m-n+1}$. Thus $\bi(m,n)=\bi(\al(m,n))$, and $L$ is the cuspidal module corresponding to the positive root $\al(m,n)$. In terms of the shape notation developed in \cite[\S3.3]{KRhomog}, this module corresponds to the hook  skew shape with the bottom tip of the hook on the runner $i_1$. 

\subsection{\boldmath Type $B$}\label{SSB}
The set $\De_+$ of positive roots is 
$$\{\al(m,n)\mid 1\leq m\leq n\leq \ell\}\cup\{\beta(m,n)\mid 1\leq m< n\leq \ell\},$$
where 
\begin{align*}
\al(m,n):=\sum_{k=m}^{n}\al_k,\quad 
\beta(m,n):=\sum_{k=1}^{m}2\al_k+\sum_{k=m+1}^n\al_k.
\end{align*}
The corresponding good Lyndon words are \cite[\S8.2]{Lec}:
\begin{align*}
\bi({\al(m,n)})=(m,m+1,\dots,n),\quad
\bi({\beta(m,n)})=(1,\dots,m,1,\dots,n),
\end{align*}
and, recalling that the quantum shuffle product used in \cite{Lec} is the opposite to the one used here, the corresponding dual canonical bases elements are \cite[\S8.2]{Lec}:
\begin{align*}
b^*_{\bi({\al(m,n)})}&=(m,m+1,\dots,n),\\
b^*_{\bi({\beta(m,n)})}&=(q+q^{-1})(1)\big[(1,\dots,n)\circ(2,\dots,m)\big].
\end{align*}

If $\al\in\De_+$ is of the form $\al(m,n)$, the corresponding cuspidal module $L_\al$ is  $1$-dimensional:
$
L_\al:=\FF\cdot v_\al,
$
with the following action of the generators:
$$
e(\bj)v_{\al}=\de_{\bj,\bi(\al)}v_{\al},\quad \psi_rv_{\al}=y_sv_{\al}=0.
$$

Let $\al=\be(m,n)$ for $1\leq m<n\leq \ell$, and set $d:=m+n=\height(\al)$. 

We first consider the case $m=1$. In this case define $L_{\al}:=\FF\cdot v_1\oplus\FF\cdot v_{-1}$ with $\deg(v_a)=a$, and 
\begin{align*}
&e(\bj)v_a=\de_{\bj,(1,1,2,\dots,n)}v_a\qquad (a=\pm1);\\
&y_rv_1=0\qquad (1\leq r\leq d);\\
&y_1v_{-1}=-v_1,\ y_2v_{-1}=v_1,\ y_sv_{-1}=0 \qquad (3\leq s\leq d);
\\
&\psi_1v_1=v_{-1},\ \psi_rv_1=0\qquad(2\leq r<d),\qquad \psi_sv_{-1}=0\qquad(1\leq s<d). 
\end{align*}
The relations are now easy to check. 

Finally consider the case $\al=\be(m,n)$ for $m\geq 2$. Denote 
$\beta:=\al(1,n)$ and $\ga:=\al(2,m)$. 
Note that the cuspidal module $L_{\al_1+\beta}=L_{\beta(1,n)}$ has been constructed above. It is $2$-dimensional with basis $\{v_1,v_{-1}\}$. The cuspidal module $L_\ga$ has also been constructed above. It is $1$-dimensional with basis $\{v\}$. Let $S_{d-1}':=\{w\in S_d\mid w(1)=1\}$, and $S'_{n,m-1}< S'_{d-1}$ be the  subgroup 
$$\{w\in S_{d-1}'\mid w(k)\leq n+1,w(l)\geq n+2\ \text{for all $k\leq n+1,l\geq n+2$}\}.$$ Let $C$ be the set of the minimal length representatives for $S'_{d-1}/S'_{n,m-1}$. 
Set
$$
L_\al:=\Ind_{\al_1,\beta,\ga}^{\al_1,\be+\ga}(\Res^{\al_1+\be}_{\al_1,\be}L_{\al_1+\be})\boxtimes L_\ga.
$$
Then $L_\al$ has basis $\{\psi_w\otimes v_a\otimes v\mid w\in C,\ a=\pm1\}$. Now extend $L_\al$ from $R_{\al_1,\be+\ga}$-module to $R_\al$-module as follows. For all $w\in C$ and $t=\pm 1$ set:
\begin{align*}
e(\bj) (\psi_w\otimes v_a\otimes v)&=0\qquad(\bj\in\words^\al,j_1\neq 1);
\\
\psi_1 (\psi_w\otimes v_a\otimes v)&=
\left\{
\begin{array}{ll}
\psi_w\otimes v_{-1}\otimes v &\hbox{if $a=1$ and $w(2)=2$,}\\
0 &\hbox{otherwise.}
\end{array}
\right.
\end{align*}
We only need to check the relations (\ref{R3YPsi}) with $s=1$, and (\ref{R5})--(\ref{R7}) with $r=1$. 
%, (\ref{R6}), (\ref{R4}), (\ref{R3Psi}), 
 
Let $r\geq 3$ and note that for $a=\pm 1$ we can write 
$$y_r\psi_w\otimes v_a\otimes v=\sum c_{u}\psi_u\otimes v_a\otimes v\qquad(c_{u}\in\FF),
$$
so that if $w(2)=2$ then the summation is over all $u\in C$ with $u(2)=2$, and if $w(2)\neq 2$ then the summation is over all $u\in C$ with $u(2)\neq 2$. 
Now (\ref{R3YPsi}) with $s=1$ follows. 

Let $r=1$. The relations (\ref{R6})--(\ref{R3Psi}) are then clear, and it remains to verify (\ref{R7}). Let $\bk=(k_1,\dots,k_d)=w(1,1,2,\dots,n,2,\dots,m)$ be the weight of $\psi_w\otimes v_a\otimes v$. If $k_3\neq 1$ then we clearly have 
$$
\psi_1\psi_2\psi_1(\psi_w\otimes v_a\otimes v)=\psi_2\psi_1\psi_2(\psi_w\otimes v_a\otimes v)=0,
$$
so we may assume that $k_3=1$. Then automatically $k_2=2$, and 
$\psi_1\psi_2\psi_1(\psi_w\otimes v_a\otimes v)=0$. We can choose reduced decompositions so that $\psi_w=\psi_2\psi_u$ for $u(2)=2$. Then
\begin{align*}
\psi_2\psi_1\psi_2(\psi_w\otimes v_a\otimes v)&=\psi_2\psi_1(\psi_2\psi_2\psi_u\otimes v_a\otimes v)
\\
&=\psi_2\psi_1((y_2^2-y_3)\psi_u\otimes v_a\otimes v)
\\
&=-\psi_2\psi_1(-y_3\psi_u\otimes v_a\otimes v).
\end{align*}
Finally, $y_3\psi_u\otimes v_a\otimes v$ can be written as a linear combination of elements of the form $\psi_x\otimes v_a\otimes v$ for $x\in C$ with $x(2)=2$ and $x(3)=3$. But $\psi_2\psi_1$ annihilates any such element. So the left hand side of (\ref{R7}) is zero. 
On the other hand, 
$$
\frac{Q_{k_1,k_2}(y_3,y_2)-Q_{k_1,k_2}(y_1,y_2)}{y_3-y_1}=\frac{(y_3^2-y_2)-(y_1^2-y_2)}{y_3-y_1}=y_1+y_3.
$$
Furthermore, using $k_3=1$ and $k_2=2$, we conclude that 
$(y_1+y_3)(\psi_w\otimes v_a\otimes v)=0$.
The relation (\ref{R7}) has been checked.

\subsection{\boldmath Type $C$}\label{SSC}
The set $\De_+$ of positive roots is 
$$\{\al(m,n)\mid 1\leq m\leq n\leq \ell\}\cup\{\beta(m,n)\mid 2\leq m\leq n\leq \ell\},$$
where 
\begin{align*}
\al(m,n):=\sum_{k=m}^{n}\al_k,\quad 
\beta(m,n):=\al_1+\sum_{k=2}^{m}2\al_k+\sum_{k=m+1}^n\al_k.
\end{align*}
The corresponding good Lyndon words are \cite[\S8.3]{Lec}:
\begin{align*}
\bi({\al(m,n)})=(m,m+1,\dots,n),\quad 
\bi({\beta(m,n)})=(1,\dots,n,2,\dots,m),
\end{align*}
and %, recalling that the quantum shuffle product used in \cite{Lec} is the opposite to the one used here, 
the corresponding dual canonical basis elements are \cite[\S8.3]{Lec}:
\begin{align*}
b^*_{\bi({\al(m,n)})}&=(m,m+1,\dots,n),\\
b^*_{\bi({\beta(m,n)})}&=q^{\de_{m,n}}\,(1)\big[(2,\dots,n)\circ(2,\dots,m)\big].
\end{align*}

If $\al\in\De_+$ is of the form $\al(m,n)$, the corresponding cuspidal module $L_\al$ is  $1$-dimensional:
$
L_\al:=\FF\cdot v_\al,
$
with 
$
e(\bj)v_{\al}=\de_{\bj,\bi(\al)}v_{\al},\quad \psi_rv_{\al}=y_sv_{\al}=0.
$

Let $\al=\be(m,n)$ for $2\leq m\leq n\leq \ell$ and set $d:=m+n-1=\height(\al)$. Consider the subalgebra $R_{\al_1,\al-\al_1}\subset R_\al$ and % generated by $\psi_2,\dots,\psi_{d-1},y_2,\dots,y_d$ and all $e(\bj)$ such that $\bj=(j_1,\dots,j_d)\in\words^\al$ with $j_1=1$. It is easy to see using Theorem~\ref{TBasis} that $R'\cong R_{\al-\al_1}$. Let us identify $R'$ and $R_{\al-\al_1}$, and 
 the $R_{\al_1,\al-\al_1}$-module  $L_{\al_1}\boxtimes (L_{\al(2,n)}\circ L_{\al(2,m)})\langle \de_{m,n}\rangle$ with the graded character $q^{\de_{m,n}}(1)\big[(2,\dots,n)\circ(2,\dots,m)\big]$. Extend the action from $R_{\al_1,\al-\al_1}$ to $R_\al$ so that $\psi_1$  and $e(\bj)$  for $\bj\in\words^\al$ with $j_1\neq 1$ act by $0$. To show that we indeed get an $R_\al$-module, we need to check the defining relations for $R_\al$ from section~\ref{SSDefKLR}. The relations involving the generators of $R_{\al_1,\al-\al_1}$ are obviously satisfied and  the relation (\ref{R4}) with $r=1$ is the only new relation which is not immediately obvious. To check it one has to note that $y_1-y_2^2$ acts as $0$ on  $L_{\al_1}\boxtimes (L_{\al(2,n)}\circ L_{\al(2,m)})\langle \de_{m,n}\rangle$. 
 % by degree considerations. 
 Thus, we have constructed a module $L_\al$ with graded character $q^{\de_{m,n}}\,(1)\big[(2,\dots,n)\circ(2,\dots,m)\big]$.

\subsection{\boldmath Type $D$}
The set $\De_+$ of positive roots is 
$$\{\al(m,n)\mid 2\leq m\leq n\leq \ell\}\cup\{\beta(m)\mid 2\leq m\leq \ell\}\cup\{\gamma(m,n)\mid 2\leq m<n\leq\ell\},$$
where 
\begin{align*}
\al(m,n)&:=\sum_{k=m}^{n}\al_k,\qquad 
\beta(m):=\al_1+\sum_{k=3}^{m}\al_k,\\
\gamma(m,n)&:=\al_1+\al_2+2\sum_{k=3}^m\al_k+\sum_{k=m+1}^n\al_k.
\end{align*}
The corresponding good Lyndon words are as follows \cite{LR}:
\begin{align*}
\bi({\al(m,n)})&=(m,m+1,\dots,n),\qquad 
\bi({\beta(m)})=(1,3,4,\dots,m),\\
\bi({\gamma(m,n)})&=(1,3,4,\dots,n,2,3,\dots,m).
\end{align*}

Assume first that $\al\in\De_+$ is of the form $\al(m,n)$ or $\beta(m)$. The corresponding cuspidal module $L_\al$ is $1$-dimensional:
$
L_\al:=\FF\cdot v_\al,
$
with 
$$
e(\bj)v_{\al}=\de_{\bj,\bi(\al)}v_{\al},\quad \psi_rv_{\al}=y_sv_{\al}=0.
$$

Now let $2\leq m<n\leq\ell$ and $\al=\ga(m,n)$. By Lemma~\ref{LLec}, the corresponding cuspidal module $L_\al$ is homogeneous and can be constructed explicitly using Theorem~\ref{Thomog}.  
By \cite[Lemma 55]{Lec} and Lemma~\ref{LCuspCan}, 
$$
\CH L_\al=(1)((3,\dots,n)\circ(2,\dots,m)-q(3,\dots,m)\circ(2,\dots,n)).
$$
In particular, 
$$
\qdim L_{\al}={m+n-3\choose m-1}-{m+n-3\choose m-2}={m+n-3\choose m-2}\frac{n-m}{m-1}.
$$
This formula can also be deduced using the Peterson-Proctor hook formula, see Theorem~\ref{Thomog}(iii). 
For example in type $D_5$, the highest root $\al$ is $\ga(4,5)$. The corresponding cuspidal module is $5$-dimensional with basis $\{v_{\bj}\mid \bj\in X\}$ where 
\begin{align*}
X:=\{(1,3,4,5,2,3,4),\ (1,3,4,2,5,3,4),\ (1,3,2,4,5,3,4),\\ 
(1,3,4,2,3,5,4),\ (1,3,2,4,3,5,4)\}.
\end{align*}
The action of the generators is as follows: $e(\bi)v_{\bj}=\de_{\bi,\bj}v_\bj$, $y_rv_{\bj}=0$, $\psi_rv_\bj=v_{s_r\bj}$ if $s_r\bj\in X$ and $\psi_rv_\bj=0$, otherwise. 
%for all admissible $\bi,\bj$ and $r$. 

\subsection{\boldmath Types $E$}
As we have already pointed out, with the exception of twelve positive roots $\al\in\mathcal E$ defined in (\ref{EExc}), the cuspidal modules $L_\al$ are homogeneous, and Theorem~\ref{Thomog} provides a construction. Then in particular $\dim L_\al=|G_\al|$. This dimension can also be often computed using the Peterson-Proctor hook formula of Theorem~\ref{Thomog}(iii) since most of the good Lyndon words turn out to be strongly homogeneous. For example, the good Lyndon word corresponding to the root
$$
\al:=\al_1+3\al_2+3\al_3+5\al_4+4\al_5+3\al_6+2\al_7+\al_8
$$
is 
$$
(1,3,4,5,6,7,8,2,4,5,6,7,3,4,5,6,2,4,5,3,4,2),
$$
which is strongly homogeneous. Now Theorem~\ref{Thomog}(iii) gives 
\begin{align*}
\dim L_\al&=\frac{22!}{1\cdot 2\cdot 3\cdot 3\cdot 4\cdot 5\cdot 4\cdot 5\cdot 6\cdot 7\cdot 5\cdot 6\cdot 7\cdot 8\cdot 9\cdot 6\cdot 7\cdot 8\cdot 9\cdot 10\cdot 11\cdot 22}
\\
&=33592.
\end{align*}

\subsection{\boldmath Type $F_4$} 
The good Lyndon words for the natural ordering are given in \cite{LR}. 
If a positive root $\al$ in $F_4$ is supported on a proper Dynkin subdiagram of $F_4$ the construction of the cuspidal module $L_\al$ reduced to the corresponding smaller rank type, and hence is known from \S\ref{SSA}--\ref{SSC}. If $\al=\al_{1}+\al_2+\al_3+\al_4$, then the cuspidal module $L_\al$ is $1$-dimensional and is constructed as in type $A_4$. We are left with exactly nine roots which are not supported on a proper Dynkin subdiagram and which have at least one coefficient greater than $1$. 

\subsection{\boldmath Type $G_2$}
In this section we will use the notation $[n]_i$ for $i=1,2$ introduced in \S\ref{SSGRP}. For example $[2]_2=q^3+q^{-3}$. 

The set $\De_+$ of positive roots is 
$$\{\al_1,\al_2,\al_1+\al_2,2\al_1+\al_2,3\al_1+\al_2,3\al_1+2\al_2\},$$
the corresponding good Lyndon words are \cite{LR}:
\begin{align*}
&\bi(\al_1)=(1),\ \bi(\al_2)=(2),\ \bi(\al_1+\al_2)=(1,2),\ \bi(2\al_1+\al_2)=(1,1,2),\\
&\bi(3\al_1+\al_2)=(1,1,1,2),\ \bi(3\al_1+2\al_2)=(1,1,2,1,2),
\end{align*}
and the corresponding dual canonical bases elements are \cite[\S5.5.4]{Lec}:
\begin{align*}
&b^*_{\bi(\al_1)}=(1),\ b^*_{\bi(\al_2)}=(2),\ b^*_{\bi(\al_1+\al_2)}=(1,2),\\ 
&b^*_{\bi(2\al_1+\al_2)}=[2]_1(1,1,2),\ 
b^*_{\bi(3\al_1+\al_2)}=[2]_1[3]_1(1,1,1,2),\\ &b^*_{\bi(3\al_1+2\al_2)}=[2]_1[3]_1(1,1,2,1,2)+[2]_1[3]_1[2]_2(1,1,1,2,2).
\end{align*}
We now exhibit the corresponding cuspidal modules. In all cases below a direct check shows that the prescribed action of generators on $L_\al$ satisfies the defining relations of the corresponding KLR algebra. It will also be clear that $\CH L_\al=b^*_{\bi(\al)}$, and so it follows from Lemma~\ref{LCuspCan} that $L_\al$ is indeed a cuspidal irreducible module corresponding to the positive root $\al$. 

For $\al\in \{\al_1,\al_2,\al_1+\al_2\}$ the cuspidal module $L_\al$ is  $1$-dimensional:
$
L_\al:=\FF\cdot v_\al,
$
with 
$
e(\bj)v_{\al}=\de_{\bj,\bi(\al)}v_{\al},\quad \psi_rv_{\al}=y_sv_{\al}=0.
$

Define $L_{2\al_1+\al_2}:=\FF\cdot v_1\oplus\FF\cdot v_{-1}$ with $\deg(v_a)=a$ and 
\begin{align*}
&e(\bj)v_a=\de_{\bj,(1,1,2)}v_a\qquad (a=\pm1);\\
&y_1v_1=y_2v_1=y_3v_1=0,\quad y_1v_{-1}=-v_1,\ y_2v_{-1}=v_1,\ y_3v_{-1}=0;
\\
&\psi_1v_1=v_{-1},\ \psi_2v_1=0,\quad \psi_1v_{-1}=0,\ \psi_2v_{-1}=0. 
\end{align*}

To deal with the remaining roots, recall Basis Theorem~\ref{TBasis} involving a choice of elements $\psi_w$. Here and below we will always choose them inductively as follows. Let $S_{d-1}<S_d$ be the subgroup generated by $s_1,\dots,s_{d-2}$. Then every element $w\in S_d$ can be written in the form $s_rs_{r+1}\dots s_{d-1}u$ for unique $1\leq r\leq d$ and $u\in S_{d-1}$. Now take $\psi_w:=\psi_r\psi_{r+1}\dots \psi_{d-1}\psi_u$.

Let $\al:=3\al_1+\al_2$, $e:=e(1,1,1,2)$, and 
$$B:=\{\psi_wy_1^{m_1}\dots y_4^{m_4}e(\bj)\mid w\in S_4,\ m_1,\dots,m_4\in\Z_{\geq 0},\ \bj\in\words^\al\}$$ 
be the standard basis of $R_\al$. Set $X:=\{\psi_ue\mid u\in S_3\}$ and $E:=B\setminus X$. We claim that the linear span $J$ of the set $E$ is a left ideal in $R_\al$. To check this, one needs to show that $xb\in J$ for any standard generator $x$ of $R_\al$ and any $b=\psi_w y_1^{m_1}\dots y_4^{m_4}e(\bj)\in E$. %Let $y=\psi_w y_1^{m_1}\dots y_4^{m_4}e(\bj)$. 
Since the standard basis is homogeneous with respect to the two-sided weight decomposition $R_\al=\bigoplus_{\bi,\bj\in \words^\al}e(\bi)R_\al e(\bj)$ and  $X\subset eR_\al e$ we may assume that $xb\in eR_\al e$. If $m_1+\dots+m_4>0$ then it is easy to see using relations that $xb$ is again a linear combination of elements of $B$ containing some positive degrees of $y$'s, so we may also assume that $m_1+\dots+m_4=0$. It follows that $b$ is of the form $ \psi_r\dots \psi_{3}\psi_ue$ for $1\leq r\leq 3$ and $u\in S_3$. The only case where $xb\in eR_\al e$ is where  $b=\psi_3\psi_ue$ and  $x=\psi_3$. 
Then 
$
xb=(y_3^3-y_4)\psi_ue,
$
which is easily checked to be a linear combination of elements in $E$. Now set $L_\al:=(R_\al/J)\langle 3\rangle$. This $R_\al$-module has graded character $[2]_1[3]_1(1,1,1,2)$. 

Finally, let $\al=3\al_1+2\al_2$ , $e:=e(1,1,1,2,2)$, $f:=e(1,1,2,1,2)$, and 
$$B=\{\psi_wy_1^{m_1}\dots y_5^{m_5}e(\bj)\mid w\in S_5,\ m_1,\dots,m_5\in\Z_{\geq 0},\ \bj\in\words^\al\}$$ 
be the standard basis of $R_\al$. Set $X:=\{\psi_ue,\psi_4\psi_ue,\psi_3\psi_4\psi_ue\mid u\in S_3\}$ and $E:=B\setminus X$. We claim that the span $J$ of the set $E$ is a left ideal in $R_\al$. To check this, one needs to show that $xb\in J$ for any standard generator $x$ of $R_\al$ and any $b=\psi_w y_1^{m_1}\dots y_5^{m_5}e(\bj)\in E$. 
%Let $=\psi_w y_1^{m_1}\dots y_5^{m_5}e(\bj)$. 
Since the standard basis is homogeneous with respect to the decomposition $R_\al=\bigoplus_{\bi,\bj\in \words^\al}e(\bi)R_\al e(\bj)$ and  $X\subset eR_\al e\oplus fR_\al e$ we may assume that $xb\in eR_\al e\oplus fR_\al e$. If $m_1+\dots+m_5>0$ then $xb$ is a linear combination of elements of $B$ containing positive degrees of $y$'s, so we may assume that $m_1+\dots+m_5=0$. Then $b$ is of the form 
$\psi_s\dots\psi_4\psi_r\dots \psi_{3}\psi_ue, $ for $1\leq r\leq 4$, $1\leq s\leq 5$, $u\in S_3$, and $(r,s)\neq (4,5),(4,4),(4,3)$. Now, degree and weight considerations show that we only need to consider the cases where $x=\psi_2,b=\psi_2\psi_3\psi_4\psi_ue$ or $x=\psi_4,b=\psi_3\psi_4\psi_3\psi_ue$. In the first case  
$
xb=(y_2^3-y_3)\psi_3\psi_4\psi_ue,
$
which is easily checked to be a linear combination of elements in $E$. In the second case  $xb=(y_4^3-y_5)\psi_3\psi_4\psi_ue$ which again belongs to $J$.  
Now set $L_\al:=(R_\al/J)\langle 6\rangle$. This $R_\al$-module has graded character $[2]_1[3]_1(1,1,2,1,2)+[2]_1[3]_1[2]_2(1,1,1,2,2)$.

\end{document}